\documentclass[a4paper,reqno,11pt]{amsart}
\usepackage{amsmath,amsthm,amsfonts,amssymb,color,enumitem,bbm,dsfont,cancel}
\usepackage[colorlinks,citecolor=red,urlcolor=red]{hyperref}

\theoremstyle{definition}
\newtheorem{theorem}{Theorem}[section]

\newtheorem{proposition}[theorem]{Proposition}
\newtheorem{lemma}[theorem]{Lemma}
\newtheorem{definition}[theorem]{Definition}
\newtheorem{remark}[theorem]{Remark}

\numberwithin{equation}{section}

\def\cB{\mathcal{B}}
\def\cD{\mathcal{D}}

\def\cF{\mathcal{F}}
\def\cH{\mathcal{H}}

\def\cN{\mathcal{N}}

\def\D{\mathbb{D}}

\def\e{\varepsilon}
\def\E{\mathbb E}

\newcommand{\R}{\mathbb{R}}
\newcommand{\N}{\mathbb{N}}
\newcommand{\W}{\dot{W}}
\newcommand{\ud}{\ensuremath{ \mathrm{d}} }
\newcommand{\Ceil}[1]{\left\lceil #1 \right\rceil}

\newcommand*{\one}{{{\rm 1\mkern-1.5mu}\!{\rm I}}}

\begin{document}

\title[Stochastic fractional diffusion equations]{Stochastic fractional diffusion equations with  Gaussian noise rough in space}

\author[Y. Guo]{Yuhui Guo}
\address{Y. Guo:
 School of Mathematics, Shandong University, Jinan, Shandong, 250100, China.}
 \email{\url{guoyuhui@mail.sdu.edu.cn}}

\author[J. Song]{Jian Song}
\address{J. Song: Research Center for Mathematics and Interdisciplinary Sciences, Shandong University, Qingdao, Shandong, 266237, China.}
 \email{\url{txjsong@sdu.edu.cn}}

\author[X. Song]{Xiaoming Song}
\address{X. Song: Department of Mathematics, Drexel University, Philadelphia, PA 19104, USA. }
 \email{\url{xs73@drexel.edu}}

\subjclass[2010]{ Primary 60H15, Secondary 60G22, ~60G60, ~26A33, ~60H07}

\keywords{Stochastic  partial differential equation,
fractional Brownian motion,
Malliavin calculus,
Mittag-Leffler function,
moment estimates,
H\"older continuity}

\date{}
\begin{abstract}
  In this article, we consider the following stochastic  fractional diffusion equation
  \begin{equation*}
    \left(\partial^{\beta}+\dfrac{\nu}{2}\left(-\Delta\right)^{\alpha / 2}\right) u(t, x)= \lambda\: I_{0_+}^{\gamma}\left[u(t, x) \dot{W}(t, x)\right] ,\quad t>0,\: x \in \mathbb{R},
  \end{equation*}
  where $\alpha>0$, $\beta\in(0,2]$, $\gamma \ge 0$, $\lambda\neq0$, $\nu>0$, and $\dot{W}$ is a Gaussian noise which is white or fractional in time and rough in space. We prove the existence and uniqueness of the solution in the It\^o-Skorohod sense and obtain the  lower and upper bounds for  the $p$-th moment. The H\"older regularity of the solution is also studied.
\end{abstract}

\maketitle{
\hypersetup{linkcolor=black}
\tableofcontents
}

\section{Introduction}

Consider the stochastic space-time fractional diffusion equations (FDEs):
\begin{equation}\label{e:fde}
	\begin{cases}
		\left(\partial^{\beta}+\dfrac{\nu}{2}\left(-\Delta\right)^{\alpha / 2}\right) u(t, x)= \lambda\: I_{0_+}^{\gamma}\left[u(t, x) \dot{W}(t, x)\right] , & t>0, x \in \mathbb{R},   \\
		u(0, \cdot)=\mu_0,                                                                                                                                                           & \text { if } \beta \in(0,1], \\
		u(0, \cdot)=\mu_0, \quad \dfrac{\partial}{\partial t} u(0, \cdot)=\mu_1,                                                                                                      & \text { if } \beta \in(1,2],
	\end{cases}
\end{equation}
where $\alpha>0$, $\beta\in(0,2]$, $\gamma\ge 0$, $\lambda\ne 0$, $\nu>0$, and $\W$ is a Gaussian noise.  In \eqref{e:fde} we use $\partial^\beta$ to denote the {\em Caputo fractional differential} operator:
\begin{equation}\label{e:CFDO}
  \partial^{\beta} f(t):=
\begin{cases}
  { \displaystyle   \dfrac{1}{\Gamma(n-\beta)}\int_{0}^{t}\frac{f^{(n)}(\tau)}{(t-\tau)^{\beta+1-n}}\ud\tau}, & \mbox{if } \beta \neq n, \\  \\
  \dfrac{\ud^n}{\ud t^n}f(t),                                                                                    & \mbox{if }  \beta = n,
\end{cases}
\end{equation}
where $n=\lceil\beta\rceil$ is the smallest integer not smaller than $\beta$ and
 $\Gamma(x) = \int_{0}^{\infty}e^{-t}t^{x-1}\ud t$ is the gamma function. We denote by
$I_{0_+}^\gamma$  the Riemann-Liouville integral in the time variable:
\begin{equation*}
  (I_{0+}^{\gamma}f)(t):=\frac{1}{\Gamma(\gamma)} \int_{0}^{t}f(r)(t-r)^{\gamma-1}\ud r.
\end{equation*}

For any $H\in(0, 1)$, let $\Lambda_H(x)=\frac12 (|x|^{2H})''$ denote half of the second derivative of $|x|^{2H}$ in the sense of distribution.  Let $B^H(t)$ be a fractional Brownian motion with Hurst parameter $H\in(0,1)$ and $\dot B^H(t) := \frac{d}{dt}B^H(t)$. Then we have  $\E[\dot B^H(t) \dot B^H(s)]= \Lambda_H(t-s)$. In particular,
\begin{equation*}
\Lambda_H(t)=\begin{cases} H(2H-1) |t|^{2H-2}, & \text{ if } H>1/2,\\
\delta(t), & \text{ if } H=1/2.
\end{cases}
\end{equation*}
We remind that $\Lambda_H$ is  nonnegative for $H\in[1/2, 1)$  while it is not nonnegative anymore for $H\in(0,1/2)$, which makes the treatments for the two cases  different. Throughout this article, we assume that the Gaussian noise $\W$ has covariance
\begin{equation*}
  \E\left[\dot{W}(t,x)\dot{W}(s,y)\right]=\Lambda_{H_0}(t-s)\Lambda_H(x-y),
\end{equation*}
with $H_0\in[1/2,1)$ and $H\in(0, 1/2)$, that is, we consider the case that $\W$ is   \emph{white} (i.e., $H_0=1/2$) or \emph{fractional} (i.e., $H_0\in(1/2,1)$) in time  and \emph{rough} in space (i.e., $H\in(0, 1/2)$).

 In this paper,  we aim to study FDEs \eqref{e:fde} with $\alpha>0$, $\beta\in(0,2]$, $\gamma\ge 0$, $\lambda\ne 0$, $\nu>0$, and $\W$ being a Gaussian noise that is white/fractional in time and rough in space. Below we briefly recall recent developments on FDEs as well as stochastic heat equations (SHEs) and stochastic wave equations (SWEs), which by no means is complete.

The FDEs \eqref{e:fde} driven by various types of Gaussian noise have been studied in literature. In Chen {\em et al.} \cite{chen2017space}, the existence and uniqueness of the solution, and the moment bounds of the solution were obtained, when $\dot W$ is a  Gaussian noise which is not rough in time or space, and $\alpha\in(0,2]$, $\beta\in(1/2,2)$, $\gamma=0$. Chen {\em et al.} \cite{chen2019nonlinear} investigated FDEs with space-time white noise when $\alpha\in(0,2]$, $\beta\in(0,2)$, $\gamma\ge0$ and also obtained  the sample path regularity of the solution. Then, Chen {\em et al.} \cite{chen2022moments} extended some corresponding results to the case for all $\alpha>0$, $\beta\in(0,2]$ and $\gamma\ge0$. Recently, Chen and Eisenberg \cite{chen2022interpolating} studied  FDEs with  time-independent multiplicative Gaussian noise.  We also refer to the works \cite{chen2017nonlinear,chen2015fractional,mijena2015space} for other cases of equation~\eqref{e:fde}.

Several methods have been developed to study the existence and uniqueness of the solution of the stochastic partial differential equations driven by rough noise in time or space. For SHEs and SWEs with multiplicative Gaussian noise  white in time and rough in space, Balan {\em et al.} \cite{balan2015spdes} used  the classical method of Picard iterations to prove the existence and uniqueness of the mild solution. Hu {\em et al.} \cite{hu2017stochastic} obtained the uniqueness by a uniform estimate of stochastic convolution and then proved the existence by taking approximations obtained by regularizing the noise and using a compactness argument on a suitable space of trajectories. Instead of the above two methods, in this paper we use the approach of Malliavin calculus which was developed  in the study of SHEs by Hu and Nualart \cite{hu2009stochastic}. We refer to \cite{balan2017intermittency,hu2016parabolic,hns11,song2020fractional} for more related results on  SHEs and SWEs.

 The intermittency property  was obtained by Bertini and Cancrini \cite{bc95} for SHEs with space-time white noise  (see also \cite{Kh14}). For SHEs with fractional noise, the large time behavior was characterized by Song \cite{song12} and the exact Lyapunov exponents were obtained by  Chen {\em et al.}~\cite{chsx15} using Feynman-Kac formulae and large deviation techniques (see also \cite{chss18} for fractional SHEs).  For SHEs with noise rough in time or space, Hu {\em et al.} \cite{hu2017stochastic}, Huang {\em et al.} \cite{huang2017large1,huang2017large2}, and Chen \cite{chen2019parabolic,chen2020parabolic}   obtained the Lyapunov exponents. The intermittency property of SWEs is more involved due to the lack of Feynman-Kac formula. For SWEs with noise that is white in time and white/colored in space, the exact second order Lyapunov exponents  were obtained by  Balan and Song \cite{balan2019second}, and for SWEs with time-independent noise, the Lyapunov exponents were  obtained by \cite{bcc22}.  We also refer to
 Balan {\em et al.} \cite{balan2017intermittency} for the intermittency property of SWEs with noise rough in space. Recently, the matching moment bounds for a class of SPDEs, including in particular SWEs, were obtained by Hu {\em et al.} \cite{hu2022matching} for  space-time colored noise and by Chen {\em et al.} \cite{chen2022moments} for space-time white noise, both of which employ Feynman diagram formula.

The H\"older continuity for SHEs and SWEs has been fully investigated in literature. We refer to \cite{balan2019holder,balan2017hyperbolic,hu2015stochastic,song2017class,song2020fractional} and the references therein. In contrast, the H\"older continuity for  FDEs is more involved and was firstly investigated by Chen {\em et al.} \cite{chen2019nonlinear} for the space-time white noise case with $\alpha\in(0,2]$, $\beta\in(0,1]$, $\gamma\in[0,1-\beta]$; then, Chen and Hu \cite{chen2022holder} extended this result to $\beta\in(0,2)$, $\gamma\ge0$ using  local fractional derivative.

To conclude the introduction, we provide some comments on our results. Firstly, in Theorem~\ref{th:DL},  the condition for the existence and uniqueness of the solution is consistent  with the existing results when \eqref{e:fde} reduces to classical equations such as SHE or SWE (see Remark \ref{re:dl-compare}).  Secondly, due to the lack of the Feynman-Kac formula, the precise lower bound for $p$th moment of the solution to \eqref{e:fde} is still an open problem, and  we provide a preliminary result based on the moment method (see \eqref{low-b}). It is expected the optimal lower bounds which match the upper bounds \eqref{up-b} may be obtained by using the techniques developed in \cite{hu2022matching, chen2022moments}. Finally, we find a new and simple method to study the H\"older continuity of the solution. See  Theorem \ref{th:holder} as well as  Remark \ref{re:holder-compare} and Remark \ref{rem:explanation} for more details on our result and method.

This paper is organized as follows. In Section \ref{se:pre}, we present some preliminaries on Malliavin calculus for  Gaussian noise $\dot W$ and then we recall some basic facts of Mittag-Leffler functions. The existence and uniqueness of the solution and the $p$-th moment bounds are obtained in Section \ref{se:ex-un}. In Section \ref{se:holder}, we derive the H\"older continuity of the solution in time and space. Finally, some technical results  are listed in Appendix~\ref{ap:lemma}.

\bigskip
Throughout this paper, we use $C$ to denote a generic constant that might change in different positions. $\Vert\cdot \Vert_p$ denotes the probability $L^p(\Omega)$-norm. $\lceil\cdot\rceil$ (resp. $\lfloor\cdot\rfloor$) is the ceiling (resp. floor) function. We use the convention $\N=\{ 1,2,\dots \}$.

\section{Preliminaries}\label{se:pre}

In this section, we provide some preliminaries on Malliavin calculus and Mittag-Leffler functions.

\subsection{Malliavin calculus}

In this part, we collect some basic knowledge on Malliavin calculus which will be used in the paper. We refer to \cite{nualart2006malliavin} for more details.

Let $\cD(\R_+\times\R)$ denote the space of all infinitely differentiable functions with compact support in $\R_+\times\R$ and let $\cH$ be the completion of $\cD(\R_+\times\R)$ with respect to the inner product $\langle \cdot,\cdot \rangle_\cH$, where
\begin{equation*}
  \langle \varphi,\psi \rangle_\cH := \int_{\R_+^2}\int_{\R} \cF\varphi(t,\cdot)(\xi) \overline{\cF\psi(s,\cdot)(\xi)} \Lambda_{H_0}(t-s) \mu(\ud\xi)\ud t\ud s.
\end{equation*}
Here, the so-called spectral  measure $\mu$ is given by \[\mu(\ud\xi)=c_H|\xi|^{1-2H}\ud\xi,\] with $H\in(0,\frac12)$ and $c_H=\frac{\Gamma(2H+1)\sin(\pi H)}{2\pi}$, and $\cF\varphi(t,\cdot)(\xi)$ is the Fourier transform of $\varphi$ in the space, that is,
\[
\cF\varphi(t,\cdot)(\xi)=\int_{\R}e^{-i\xi x}\varphi(t,x)dx.
\]
Let $W=\{W(\varphi):\varphi\in\cD(\R_+\times\R)\}$  be a centered Gaussian process with covariance
\begin{equation*}
  \E[W(\varphi)W(\psi)]=\langle \varphi,\psi \rangle_\cH.
\end{equation*}
Thus,  the map $\varphi\mapsto W(\varphi)$ is an isometry and can be extend to $\cH$. We say that $W(\varphi)$ is the Wiener integral of $\varphi$ with respect to $W$ which is also denoted by
\begin{equation*}
\int_{0}^{\infty}\int_{\R}\varphi(t,x)W(\ud t,\ud x):=  W(\varphi), \ \ \varphi\in\cH.
\end{equation*}

 The $n$-th Hermite polynomial is denoted by $H_n$,  and the $n$-th Wiener chaos is the closed linear span of the random variables $H_n(W(\varphi))$ for $\varphi\in\cH$  with $\Vert\varphi\Vert_\cH=1$. We denote the $n$-th Wiener chaos by $\cH_n$ and let $\cF$ be the $\sigma$-field generated by $\{W(\varphi): \varphi\in\cH\}$. Then, any random variable $F\in L^2(\Omega,\cF,P)$ has the following  Wiener chaos expansion:
\begin{equation*}
  F=\E[F]+\sum_{n\geq1}I_n(f_n),
\end{equation*}
where $f_n\in \cH^{\otimes n}$ and $I_n:\cH^{\otimes n}\rightarrow \cH_n$ is the multiple Wiener integral with respect to $W$. For any $f_n\in\cH^{\otimes n}$, we also denote
\begin{equation*}
 \int_{\R_+^n}\int_{\R^n}f_n(s_1,x_1,\dots,s_n,x_n)W(\ud s_1,\ud x_1)\dots W(\ud s_n,\ud x_n):=  I_n(f_n).
\end{equation*}
We have
\begin{equation*}
  I_n(f_n)=I_n(\tilde{f_n}),
\end{equation*}
where $\tilde{f_n}$ is the symmetrization of $f_n$ in variables $(s_i, x_i), i=1, \dots, n$:
\begin{equation}\label{e:sym}
  \tilde{f_n}(s_1,x_1,\dots,s_n,x_n)=\frac{1}{n!}\sum_{\sigma\in S_n}f(s_{\sigma(1)},x_{\sigma(1)},\dots,s_{\sigma(n)},x_{\sigma(n)}),
\end{equation}
with $S_n$ being the set of all permutations of $\{1,\dots,n\}$. We also have the following property, for any $f_n\in\cH^{\otimes n}$ and $g_m\in\cH^{\otimes m}$,
\begin{equation}\label{f1}
  \E[I_n(f_n)I_m(g_m)]=
  \begin{cases}
   n! \langle \tilde{f}_n,\tilde{g}_m \rangle_{\cH^{\otimes n}}, & \mbox{if } n=m, \\
    0, & \mbox{if } n\neq m.
  \end{cases}
\end{equation}
Then, for any random variable $F\in L^2(\Omega,\cF,P)$, by \eqref{f1} we obtain
\begin{equation*}
  \E[|F|^2]=\sum_{n\geq0}\E[|I_n(f_n)|^2]=\sum_{n\geq0}n!\Vert \tilde{f_n}\Vert_{\cH^{\otimes n}}^2.
\end{equation*}
and in our case, the norm of $\cH^{\otimes n}$ is
\begin{equation}\label{e:norm-fn}
  \Vert f_n\Vert_{\cH^{\otimes n}}^2=\int_{\R_+^{2n}}\int_{\R^n}  \cF f(\mathbf{t},\cdot,t,x)(\boldsymbol \xi) \overline{\cF f(\mathbf{s},\cdot,t,x)(\boldsymbol \xi)} \prod_{j=1}^{n}\Lambda_{H_0}(t_j-s_j) \boldsymbol\mu(\ud\boldsymbol\xi)\ud\mathbf{t}\ud\mathbf{s},
\end{equation}
where we use the convention that $\ud\mathbf{s}:=\prod_{i=1}^{n}\ud s_i$ for $\mathbf{s}:=(s_1,\dots,s_n)$ and similarly for $\mathbf{t}$, $\ud\mathbf{t}$, $\boldsymbol\xi$ and $\boldsymbol\mu(\ud\boldsymbol\xi)=\prod_{j=1}^n|\xi_i|^{1-2H}\ud \xi_i$.
For the norm $\Vert \cdot\Vert_p,\ p\ge 2$ on a fixed Wiener chaos space $\cH_n$ we have the hypercontractivity that
\begin{equation}\label{pnorm-bound}
  \Vert F\Vert_p\le (p-1)^{n/2}\Vert F\Vert_2,\quad    F\in\cH_n.
\end{equation}

Now, we introduce Malliavin derivative and divergence operator (also called Skorohod integral).  Let the space $\D^{1,2}$ be the closure of the set of smooth and cylindrical random variables of the form
$F=f\left( W(\varphi_1),\dots,W(\varphi_n)\right)$, where $\varphi_i\in\cH$, $n\geq1$ and $f$ belongs to the space $C^\infty_b(\R^n)$ of  bounded $C^\infty$-functions on $\R^n$ whose partial derivatives of all orders are bounded,  under the norm
\begin{equation*}
  \Vert DF\Vert_{1,2}=\sqrt{\E[|F|^2]+\E[\Vert DF\Vert_\cH^2]},
\end{equation*}
where
\begin{equation*}
  DF:=\sum_{i=1}^{n} \frac{\partial f}{\partial x_i}\left( W(\varphi_1),\dots,W(\varphi_n)\right)\varphi_i
\end{equation*}
is the Malliavin derivative of $F$.

The divergence operator $\delta$ is defined as the adjoint of the operator $D$ by the following duality formula:
\begin{equation*}
  \E[F\delta(u)]=\E[\langle DF,u\rangle_\cH], ~~ F\in \D^{1,2},  u\in \text{Dom}\:\delta,
\end{equation*}
where $\text{Dom}\:\delta$ is the set of $u\in L^2(\Omega,\cH)$ such that
\begin{equation*}
  \left|\E[\langle DF,u\rangle_\cH]\right|\leq c \Vert F\Vert_2, ~~ \forall F\in \D^{1,2},
\end{equation*}
 where $c$ is some constant depending on $u$.

\subsection{Mittag-Leffler function}

For $a>0, b\in\mathbb{C}$, the two-parameter Mittag-Leffler function is defined by (see, e.g., \cite[Sect.1.2]{podlubny1998fractional}):
\begin{equation}\label{e:mlf}
  E_{a, b}(z):=\sum_{k=0}^{\infty}\frac{z^k}{\Gamma(a k+b)},
\end{equation}
and we use the convention $E_a(\cdot):=E_{a,1}(\cdot)$. For each $n\in\mathbb{N}$, direct computations yield
\begin{equation}\label{e:dmlf}
  \frac{\ud^n}{\ud z^n} \left(z^{b-1} E_{a,b}(\lambda z^a)\right) = z^{b-n-1} E_{a,b-n}(\lambda z^a),
\end{equation}
for $a>0$ and $b,  \lambda\in\mathbb{C}$.

 According to (1.8.28) and (1.8.31) in \cite{kilbas2006theory}, the asymptotic behavior of $E_{a,b}(z)$ can be described in two regimes: $0<a<2$ and $a= 2$.
\begin{proposition} \label{E-asymp}
For $z\in (-\infty, 0)$  with $z\to -\infty$, we have for $n\in\N$,
 \begin{equation*}
E_{a,b}(z)=
    \begin{cases}
    \displaystyle{   -\sum_{k=1}^{n}\frac{1}{\Gamma(b-ak)\cdot z^k} +O\left(\frac{1}{z^{n+1}}\right),} &\text{if } 0<a<2,  \\
      \displaystyle {(-z)^{(1-b)/2}\cos\left(\sqrt{-z}+\frac{\pi(1-b)}{2}\right) -\sum_{k=1}^{n}\frac{1}{\Gamma(b-2k)\cdot z^k}+O\left(\frac{1}{z^{n+1}}\right)},  & \text{if } a=2.
      \end{cases}
  \end{equation*}
  \end{proposition}
\bigskip

We introduce two constants defined via Mittag-Leffler functions that will be used in the sequel. Let $ \alpha>0, \beta>0, \nu>0$ be the given coefficient constants in \eqref{e:fde}, for any $a\in \R, \gamma_1,\gamma_2\in\R$,  we denote
\begin{align}
  C_{a, \gamma_1,\gamma_2}:= & \left(2^{-1}\nu\right)^{-\frac{a+1}{\alpha}}\int_{\R}E_{\beta,\gamma_1}(-|\xi|^\alpha)E_{\beta,\gamma_2}(-|\xi|^\alpha)|\xi|^a\ud \xi ,\label{e:cgamma12}\\
  C_{a,\gamma}:= & C_{a, \gamma,\gamma}\label{e:cgamma}.
\end{align}

\section{Existence, uniqueness and moment bounds of the solution} \label{se:ex-un}

In this section, we give the definition of \emph{a mild Skorohod solution} to \eqref{e:fde}, and then we provide a necessary and sufficient condition for its existence and uniqueness. We also establish the $p$-moment bounds of the solution.

To define a solution to \eqref{e:fde}, we first   solve its deterministic counterpart:
  \begin{align} \label{e:pde}
    \begin{cases}
      \left(\partial^\beta + \dfrac{\nu}{2} (-\Delta)^{\alpha/2} \right) u(t,x)= I_{0+}^\gamma\left[f(t,x)\right], & \qquad t>0,\: x\in\R^d, \\[1em]
      \left.\dfrac{\partial^k}{\partial t^k} u(t,x)\right|_{t=0}=\mu_k(x),                                          & \qquad 0\le k\le \Ceil{\beta}-1, \:\: x\in\R^d.    \end{cases}
  \end{align}
As shown in \cite[Theorem C.1]{chen2022moments}, the solution to \eqref{e:pde} is
\begin{equation}\label{e:sol-pde}
  u(t,x) = J_0(t,x) + \int_0^t \ud s \int_{\R^d} \ud y\: f(s,y) \: Y(t-s,x-y)
\end{equation}
where $J_0(t,x)$ is the solution to the homogeneous equation (namely, equation \eqref{e:pde} with $f\equiv0$):
\begin{equation}\label{e:J0}
  J_0(t,x) =
  \begin{cases} \displaystyle
    \int_{\R} Z(t,x-y)\mu_0(y)\ud y                                       & \text{if $\beta\in(0,1]$},  \vspace{0.2cm}\\
   \displaystyle \int_{\R} Z^*(t,x-y)\mu_0(y)\ud y + \int_{\R} Z(t,x-y)\mu_1(y)\ud y   & \text{if $\beta\in(1,2]$}.
  \end{cases}
\end{equation}
For functions $Y(t,x)$,  $Z(t,x)$, and $Z^*(t,x)$
appearing in \eqref{e:sol-pde}, we have their  Fourier transforms in spatial variable:
\begin{align}
     \cF Z(t,\cdot)(\xi) =& t^{\lceil\beta\rceil-1}E_{\beta,\lceil\beta\rceil}(-2^{-1}\nu t^\beta|\xi|^\alpha), \label{e:fourier-Z}\\
     \cF Y(t,\cdot)(\xi) =& t^{\beta+\gamma-1}E_{\beta,\beta+\gamma}(-2^{-1}\nu t^\beta|\xi|^\alpha), \label{e:fourier-Y}\\
     \cF Z^*(t,\cdot)(\xi) =& E_{\beta}(-2^{-1}\nu t^\beta|\xi|^\alpha),\ \mbox{if}\ \beta\in(1,2],\label{e:fourier-Z*}
\end{align}
where the function $E_{a,b}$ is given in \eqref{e:mlf}.

Throughout this paper, we assume that  the initial data $\mu_0>0$ and $\mu_1\ge0$ are constants. Then,  by \eqref{e:J0}, \eqref{e:fourier-Z} and $\eqref{e:fourier-Z*}$ we get
\begin{equation*}
  J_0(t,x) =
  \begin{cases}
    \mu_0\cF Z(t,\cdot)(0)=\mu_0,                              & \text{if $\beta\in(0,1]$}, \\
    \mu_0\cF Z^*(t,\cdot)(0)+\mu_1\cF Z(t,\cdot)(0)=\mu_0+\mu_1t,   & \text{if $\beta\in(1,2]$},
  \end{cases}
\end{equation*}
which does not depend on $x$ and hence is denoted by $J_0(t)$ in the sequel.

Let $\cB_b(\R)$ denote the collection of Borel sets of $\R$ with finite Lebesgue measure, and let
\begin{equation*}
  \cF_t=\sigma(W_s(A):0\le s\le t, A\in\cB_b(\R))\vee\cN,\quad t\geq0,
\end{equation*}
be the natural filtration augmented by the $\sigma$-field $\cN$ generated by all $P$-null sets.

\begin{definition}\label{def:solution}
  We say that an $(\mathcal F_t)$-adapted random field $\{u(t,x), t\ge 0, x\in\R\}$ is a mild Skorohod solution to \eqref{e:fde}, if $\E[|u(t,x)|^2]<\infty$ and
 \begin{equation}\label{e:fde'}
  u(t,x) = J_0(t)+ \lambda\int_{0}^{t}\int_{\R}Y(t-s,x-y) u(s,y)W(\ud s,\ud y),
\end{equation}
 for all $t\geq0$ and $x\in \R$,
where the integral on the right-hand side is a Skorohod integral.
\end{definition}

For a mild Skorohod solution $u(t,x)$, noting that $\E[|u(t,x)|^2]<\infty$, we have its unique chaos expansion
\begin{equation}\label{e:u-chaos}
u(t,x) =J_0(t) +\sum_{n=1}^\infty I_n(f_n),  \end{equation}
where $I_n$ is the $n$th Wiener integral and    \begin{equation}\label{e:fn}
  f_n(\mathbf{s},\mathbf{x}, t, x)=  f_n(s_1,x_1,\dots,s_n.x_n,t,x)= \lambda^n \prod_{i=1}^{n} Y(s_{i+1}-s_{i},x_{i+1}-x_i)J_0(s_1)\one_{\{0<s_1<\dots<s_n<t\}}.
  \end{equation}
  Here we use the convention $s_{n+1}=t$ and $x_{n+1}=x$. Let $\tilde f_n$ be the symmetrization of $f_n$ in the sense of \eqref{e:sym}, i.e,
  \begin{equation}\label{e:gn}
  \begin{split}
 \tilde{f}_n(\mathbf{s},\mathbf{x}, t, x)   = &  \tilde{f_n}(s_1,x_1,\dots,s_n.x_n,t,x) \\
    = & \frac{\lambda^n}{n!}\sum_{\sigma\in S_n}\bigg(Y(t-s_{\sigma(n)},x-x_{\sigma(n)})\cdots Y(s_{\sigma(2)}-s_{\sigma(1)},x_{\sigma(2)}-x_{\sigma(1)})\\
    & \qquad \qquad\quad \times J_0(s_{\sigma(1)})\one_{\{0<s_{\sigma(1)}<\dots<s_{\sigma(n)}<t\}}\bigg)\\
    = & \frac{\lambda^n}{n!}\bigg(Y(t-s_{\rho(n)},x-x_{\rho(n)})\cdots Y(s_{\rho(2)}-s_{\rho(1)},x_{\rho(2)}-x_{\rho(1)}) J_0(s_{\rho(1)})\bigg),
    \end{split}
  \end{equation}
where $\rho\in S_n$ is the permutation such that $0<s_{\rho(1)}<\dots<s_{\rho(n)}<t$. In the following we will use $s_{\rho(n+1)}$ to denote $t$.

  The Fourier transform of $\tilde{f}_n$ is given by
  \begin{equation}\label{eqn-Fourier-g}
  \cF\tilde{f}_n(\mathbf{s}, \cdot, t,x)(\boldsymbol\xi)=\frac{\lambda^n}{n!}e^{-ix(\xi_1+\dots+\xi_n)} J_0(s_{\rho(1)}) \prod_{j=1}^n\cF Y(s_{\rho(j+1)}-s_{\rho(j)}, \cdot)(\xi_{\rho(1)}+\dots+\xi_{\rho(j)}).
  \end{equation}

We are ready to   state and prove the main result of this section.

\begin{theorem}\label{th:DL}
Let $H_0\in[\frac12,1)$,  $H\in(0,1/2)$, and the condition
\begin{equation}\label{e:DL}
 \begin{cases}
  2\alpha+\frac{\alpha}{\beta} \min\left(2\gamma-2+2H_0,2\gamma-1+\frac{\beta(1-2H)}{\alpha},0\right)>3-4H, & \mbox{if } \beta\in(0,2), \\
   \alpha\min(1+\gamma,2)>3-4H, & \mbox{if } \beta=2.
 \end{cases}
\end{equation}
hold.  Then equation \eqref{e:fde} has a unique mild Skorohod solution $u(t,x)$ in the sense of Definition~\ref{def:solution}.
Moreover, the solution $u(t,x)$ has  finite $p$-th moment for all $p\ge 2$, and there exist constants $C_1,C_2$ independent of $t$, $x$ and $p$, such that for all $t\ge0$, $x\in\R$ and $p\ge2$,
\begin{equation}\label{up-b}
  \Vert u(t,x)\Vert_p \le  J_0(t)C_1 \exp \left( C_2 |\lambda|^{\frac{2}{2H_0\theta+1}} p^\frac{1}{2H_0\theta+1} t^\frac{2H_0(\theta+1)}{2H_0\theta+1} \right),
\end{equation}
where we denote
\begin{equation}\label{e:theta}
  \theta:={\frac{1}{2H_0}\left(2\beta+2\gamma-2-\frac{\beta(2-2H)}{\alpha}\right)},
\end{equation}
for $\beta\in(0,2]$.

Moreover,  in the case $\beta\in(0,1]$ and  the case $\beta=2, \gamma=0$, there exist constants $c_1,c_2$ independent of $t$, $x$ and $p$, such that for all $t\ge0$, $x\in\R$ and $p\ge2$,
\begin{equation}\label{low-b}
  \Vert u(t,x)\Vert_p \ge \mu_0c_1\exp\left(  c_2|\lambda|^{\frac{2}{2H_0\theta+1}}t^\frac{2H_0(\theta+1)}{2H_0\theta+1} \right).
\end{equation}
\end{theorem}

\begin{remark}\label{rem:theta}
Note that under condition \eqref{e:DL}, we have $2H_0\theta+1>0$. More precisely, for $\beta\in(0,2)$, \eqref{e:DL} yields $2\alpha+\frac{\alpha}{\beta} \left(2\gamma-1+\frac{\beta(1-2H)}{\alpha}\right)>3-4H$ which is equivalent to $ 2H_0\theta+1>0$, and for $\beta=2$ \eqref{e:DL} implies $\alpha(1+\gamma)>3-4H$ which is equivalent to $2H_0\theta>\frac2\alpha(1-2H)>0.$
\end{remark}

\begin{proof}
  We prove the case when $H_0\in(\frac12,1)$, while the case $H_0=\frac12$ is similar and easier, so we omit the details.

  {\em Step 1. } We shall prove that under condition \eqref{e:DL}, $u(t,x)$ given by \eqref{e:u-chaos} is a solution to \eqref{e:fde}, namely
  \begin{equation*}
    \E\left[ |u(t,x)|^2\right]=\sum_{n=0}^{\infty}n!\Vert \tilde{f_n}(\cdot,t,x)\Vert_{\cH^{\otimes n}}^2<\infty.
  \end{equation*}
  We will use the convention $\boldsymbol\xi:=(\xi_1,\dots,\xi_n)$ and similarly for $\mathbf s$, $\mathbf r$, $\boldsymbol\eta$ and $\boldsymbol\mu(\ud\boldsymbol\xi)=\prod_{j=1}^n |\xi_j|^{1-2H}\ud \xi_j$ in the following proof.  Using \eqref{e:norm-fn}, \eqref{eqn-Fourier-g}, Lemma \ref{le:B.3} and the integral form of Minkowski's inequality, we see that
  \begin{align}
      & n!\Vert \tilde{f_n}(\cdot,t,x)\Vert_{\cH^{\otimes n}}^2 \notag\\
    \le& \frac{1}{n!} J_0^2(t) C^n\lambda^{2n} \bigg( \int_{[0,t]^n} \bigg( \int_{\R^n} \prod_{j=1}^{n} \Big|\cF Y(s_{\rho(j+1)}-s_{\rho(j)},\cdot)(\xi_{\rho(1)}+\dots+\xi_{\rho(j)})\Big|^2 \boldsymbol\mu(\ud \boldsymbol\xi) \bigg)^\frac{1}{2H_0} \ud\mathbf s \bigg)^{2H_0} \notag\\
    = & (n!)^{2H_0-1} J_0^2(t) C^n\lambda^{2n} \bigg( \int_{T_n(t)} \bigg( \int_{\R^n} \prod_{j=1}^{n} |\cF Y(s_{j+1}-s_j,\cdot)(\xi_1+\dots+\xi_j)|^2  \prod_{j=1}^{n} |\xi_j|^{1-2H} \ud \boldsymbol\xi \bigg)^\frac{1}{2H_0} \ud\mathbf s \bigg)^{2H_0} \notag\\
    = & (n!)^{2H_0-1} J_0^2(t) C^n\lambda^{2n} \bigg( \int_{T_n(t)} \bigg( \int_{\R^n} \prod_{j=1}^{n} |\cF Y(s_{j+1}-s_j,\cdot)(\eta_j)|^2  \prod_{j=1}^{n} |\eta_j-\eta_{j-1}|^{1-2H} \ud \boldsymbol\eta \bigg)^\frac{1}{2H_0} \ud\mathbf s \bigg)^{2H_0} \label{e:est-f-n},
  \end{align}
  where $C>0$ is a generic constant independent of $(t,x)$, $\eta_0=0$, and \[T_n(t)=\{\mathbf{s}=(s_1,s_2,\dots,s_n);0<s_1<s_2<\dots<s_n<t\}.\]


  Using the inequality $(a+b)^p\leq a^p+b^p$ for $p\in(0,1)$ and $a,b\ge0$, we have
 \begin{equation}\label{e:ineq1}
   |\eta_j-\eta_{j-1}|^{1-2H}\leq (|\eta_j|+|\eta_{j-1}|)^{1-2H}\leq |\eta_{j-1}|^{1-2H}+|\eta_j|^{1-2H}.
 \end{equation}
Also noting that for nonnegative numbers $\{a_i\}_{i\in S}$ and $\{b_i\}_{i\in S}$ with index set $S$, we have
 \begin{equation}\label{e:ineq2}
   \prod_{i\in S} (a_i+b_i) = \sum_{I\subset S} \left(\prod_{i\in I}a_i\right) \left(\prod_{i\in S/I}b_i\right).
 \end{equation}
 Hence, combing \eqref{e:ineq1} and \eqref{e:ineq2} we have
 \begin{align}\label{e:estimate-eta}
      &\prod_{j=1}^{n}|\eta_j-\eta_{j-1}|^{1-2H} =  |\eta_1|^{1-2H} \prod_{j=2}^{n}|\eta_j-\eta_{j-1}|^{1-2H} \leq |\eta_1|^{1-2H} \prod_{j=2}^{n} \left( |\eta_{j-1}|^{1-2H}+|\eta_j|^{1-2H} \right)\notag\\
        &=|\eta_1|^{1-2H} \sum_{ J\subset \{2,\dots,n\}} \bigg(\prod_{ j\in J}|\eta_{j-1}|^{1-2H}\bigg) \bigg(\prod_{j\in S/J}|\eta_j|^{1-2H}\bigg)= \sum_{a\in \cD_n} \prod_{j=1}^{n} |\eta_j|^{a_j},
 \end{align}
 where $\cD_n$ is a set of cardinality $2^{n-1}$ consisting of multi-indices $a = (a_1,\dots,a_n)$ with the following properties:
 \begin{align}\label{e:a}
      &|a| = \sum_{j=1}^{n} a_j = n(1-2H);&& a_1\in\{1-2H,2(1-2H)\}; \notag\\
      &a_j\in\{ 0,1-2H,2(1-2H) \} \text{ for } j=2,\dots,n-1;&&    a_n\in\{ 0,1-2H\}.
 \end{align}
Applying \eqref{e:estimate-eta} to \eqref{e:est-f-n} yields
 \begin{align*}
      & n!\Vert \tilde{f_n}(\cdot,t,x)\Vert_{\cH^{\otimes n}}^2 \\
      \le  & (n!)^{2H_0-1} J_0^2(t) C^n\lambda^{2n} \bigg( \int_{T_n(t)} \bigg( \int_{\R^n} \prod_{j=1}^{n} |\cF Y(s_{j+1}-s_j,\cdot)(\eta_j)|^2 \times \sum_{a\in \cD_n} \prod_{j=1}^{n} |\eta_j|^{a_j} \ud \boldsymbol\eta \bigg)^\frac{1}{2H_0} \ud\mathbf s \bigg)^{2H_0}  \\
      \le  & (n!)^{2H_0-1} J_0^2(t) C^n\lambda^{2n}  \bigg(\int_{T_n(t)}\sum_{a\in \cD_n} \bigg(\int_{\R^n} \prod_{j=1}^{n} |\cF Y(s_{j+1}-s_j,\cdot)(\eta_j)|^2 \times |\eta_j|^{a_j} \ud \boldsymbol\eta\bigg)^\frac{1}{2H_0} \ud\mathbf s\bigg)^{2H_0} .
 \end{align*}
Note that we are assuming $H\in(0, 1/2),\alpha>0, \gamma\ge 0$. On one hand,  condition \eqref{e:DL} implies
\begin{equation*}
  \begin{cases}
    2(1-2H)<2\alpha-1, & \mbox{if } \beta\in(0,2), \\
    2(1-2H)<\alpha \min(1+\gamma,2)-1, & \mbox{if } \beta=2,
  \end{cases}
\end{equation*}
which yields the condition in \eqref{e:con-cg-wd} for $ a=a_j\in\{0, 1-2H, 2(1-2H)\}$.
Then, by Lemma \ref{le:a1}, we have
\begin{equation}\label{e:estimate-fn-time}
  \begin{split}
     & n!\Vert \tilde{f_n}(\cdot,t,x)\Vert_{\cH^{\otimes n}}^2 \\
     \le  & (n!)^{2H_0-1}J_0^2(t) C^n\lambda^{2n} \left( \sum_{a\in \cD_n}\int_{T_n(t)} \prod_{j=1}^{n}|s_{j+1}-s_j|^{\frac{1}{2H_0}(2\beta+2\gamma-2-\frac{\beta(a_j+1)}{\alpha})}  \ud\mathbf s\right)^{2H_0}.
  \end{split}
\end{equation}
On the other hand, assumption \eqref{e:DL} also yields
\begin{equation*}
  \begin{cases}
    2\beta+2\gamma-2-\frac{\beta(3-4H)}{\alpha}>-2H_0, & \mbox{if } \beta\in(0,2), \\
    2+2\gamma-\frac{2(3-4H)}{\alpha}>-2H_0, & \mbox{if } \beta=2,
  \end{cases}
\end{equation*}
and hence  we have
\begin{equation}\label{e:theta'}
\frac{1}{2H_0}(2\beta+2\gamma-2-\frac{\beta(a_j+1)}{\alpha})>-1 \text{ for all } \beta\in(0,2], a_j\in\{0, 1-2H, 2(1-2H)\}
\end{equation}
Then,   by Lemma \ref{le:t}, we obtain
\begin{equation}\label{e:time}
  \int_{T_n(t)} \prod_{j=1}^{n}|s_{j+1}-s_j|^{\frac{1}{2H_0}(2\beta+2\gamma-2-\frac{\beta(a_j+1)}{\alpha})}  \ud\mathbf s \le \frac{C^nt^{n(\theta+1)}}{\Gamma(n(\theta+1)+1)},
\end{equation}
where $\theta= {\frac{1}{2H_0}\left(2\beta+2\gamma-2-\frac{\beta(2-2H)}{\alpha}\right)}$  for either $\beta\in(0,2)$ or $\beta=2$.

Plugging \eqref{e:time} into \eqref{e:estimate-fn-time} and using the fact that $\cD_n$ is a set of cardinality $2^{n-1}$, we have
\begin{equation}\label{e:estimate-nfn}
  n!\Vert \tilde{f_n}(\cdot,t,x)\Vert_{\cH^{\otimes n}}^2 \le J_0^2(t) C^n\lambda^{2n} (n!)^{2H_0-1}\left( \frac{t^{n(\theta+1)}}{\Gamma(n(\theta+1)+1)}\right)^{2H_0}.
\end{equation}
By assumption \eqref{e:DL} we can obtain $2H_0\theta+1>0$ for either $\beta\in(0,2)$ or $\beta=2$. Then, from Lemma \ref{le2}, it follows that
\begin{equation*}
 \sum_{n\ge0} n!\Vert \tilde{f_n}(\cdot,t,x)\Vert_{\cH^{\otimes n}}^2 \le J_0^2(t) \sum_{n\ge0} \frac{C^n\lambda^{2n}t^{2nH_0(\theta+1)}}{(n!)^{2H_0\theta+1}} \le J_0^2(t) C_1\exp \left( C_2 |\lambda|^{\frac{2}{2H_0\theta+1}} t^{\frac{2H_0(\theta+1)}{2H_0\theta+1}}\right)<\infty,
\end{equation*}
for some positive constants $C_1$ and $C_2$ independent of $t$ and $x$.

{\em Step 2.} Let us  prove upper bound \eqref{up-b}. By Minkowski’s inequality and \eqref{pnorm-bound}, for $p\ge 2$,
  \begin{align*}
  \Vert u(t,x)\Vert_p &\leq \sum_{n\geq0} \Vert I_n(f_n(\cdot,t,x))\Vert_p \leq \sum_{n\geq0} (p-1)^{\frac n2}\Vert I_n(f_n(\cdot,t,x))\Vert_2 \\
       &= \sum_{n\geq0} (p-1)^{\frac n2}\left(n!\Vert \tilde{f_n}(\cdot,t,x)\Vert_{\cH^{\otimes n}}^2\right)^{1/2}.
 \end{align*}
 Plugging \eqref{e:estimate-nfn} into above equation, we obtain that
  \begin{align*}
       \Vert u(t,x)\Vert_p &\le J_0(t) \sum_{n\geq0}(p-1)^{\frac n2} C^n |\lambda|^{n} (n!)^{H_0-\frac{1}{2}} \left( \frac{t^{n(\theta+1)}}{\Gamma(n(\theta+1)+1)} \right)^{H_0}\\
       & \le J_0(t) \sum_{n\geq0} \frac{C^n |\lambda|^{n} p^{\frac n2} t^{nH_0(\theta+1)}} {(n!)^{H_0\theta+\frac{1}{2}}}
       \le J_0(t) C_1 \exp \left( C_2 |\lambda|^{\frac{2}{2H_0\theta+1}} p^\frac{1}{2H_0\theta+1} t^\frac{2H_0(\theta+1)}{2H_0\theta+1} \right),
  \end{align*}
  where the last two steps follow from Lemma \ref{le2}.

  {\em Step 3.} Finally, we prove the lower bound \eqref{low-b} for  the cases $\beta\in(0,1]$ and  $\beta=2, \gamma=0$. By \eqref{e:norm-fn},  we have
  \begin{align*}
       &n!\Vert \tilde{f_n}(\cdot,t,x)\Vert_{\cH^{\otimes n}}^2  \\
       &= n!C^{n}\int_{\R^{n}}\int_{[0,t]^{2n}} \cF \tilde{f}_n(\mathbf s, \cdot, t,x)(\boldsymbol\xi) \overline{\cF \tilde{f}_n(\mathbf r, \cdot, t,x)(\boldsymbol\xi)}\prod_{j=1}^{n}|s_j-r_j|^{2H_0-2}\ud\mathbf{s}\ud\mathbf{r} \boldsymbol\mu(\ud\boldsymbol\xi)\\
       & = n!\lambda^{2n}C^{n}\int_{\R^{n}}\int_{T^2_n(t)}\prod_{j=1}^{n}\cF Y(s_{j+1}-s_j,\cdot)(\xi_1+\dots+\xi_j)J_0(s_1) \\
       &\quad \times\cF Y(r_{j+1}-r_j,\cdot)( \xi_1+\dots+\xi_j)J_0(r_1)\times |s_j-r_j|^{2H_0-2} \ud\mathbf{s}\ud\mathbf{r} \boldsymbol\mu(\ud\boldsymbol\xi)\\
       &= n!\lambda^{2n}C^{n}\int_{\R^{n}} \bigg( \int_{T^2_n(t)} \prod_{j=1}^{n}\cF Y(s_{j+1}-s_j,\cdot)(\eta_j)J_0(s_1) \times\cF Y(r_{j+1}-r_j,\cdot)(\eta_j)J_0(r_1)\\
       &\quad\times |s_j-r_j|^{2H_0-2} \ud\mathbf{s}\ud\mathbf{r}\bigg)\prod_{j=1}^{n}|\eta_j-\eta_{j-1}|^{1-2H}  \ud\boldsymbol\eta \\
       &\ge n!\lambda^{2n}C^{n}\int_{\D_n} \bigg( \int_{T^2_n(t)} \prod_{j=1}^{n}\cF Y(s_{j+1}-s_j,\cdot)(\eta_j)J_0(s_1) \times\cF Y(r_{j+1}-r_j,\cdot)(\eta_j)J_0(r_1)\\
       &\quad\times |s_j-r_j|^{2H_0-2} \ud\mathbf{s}\ud\mathbf{r}\bigg)\prod_{j=1}^{n}|\eta_j|^{1-2H}  \ud\boldsymbol\eta,
  \end{align*}
 where in the last inequality we have used the fact that the inner integral with respect to $\ud \mathbf s\ud\mathbf r$ is nonnegative and that $|\eta_j-\eta_{j-1}|^{1-2H}\ge|\eta_j|^{1-2H}$ on $\D_n$ with $\D_n=\{ (\eta_1,\dots,\eta_n)\in\R^n: \eta_1\ge0,\eta_2\le0,\eta_3\ge0,\eta_4\le0,\dots \}$.

  Now, we have
  \begin{align}
    &n!\Vert \tilde{f_n}(\cdot,t,x)\Vert_{\cH^{\otimes n}}^2  \notag\\
    & \ge n!\lambda^{2n}C^{n}\int_{T^2_n(t)} \bigg(\int_{\D_n}\prod_{j=1}^{n}\cF Y(s_{j+1}-s_j,\cdot)(\eta_j) \times\cF Y(r_{j+1}-r_j,\cdot)(\eta_j) \notag\\
    &\quad\times|\eta_j|^{1-2H}  \ud\boldsymbol\eta \bigg)J_0(s_1)J_0(r_1) \prod_{j=1}^{n}|s_j-r_j|^{2H_0-2} \ud\mathbf{s}\ud\mathbf{r} \notag\\
    &= n!\lambda^{2n}C^{n}\int_{T^2_n(t)} \bigg(\prod_{j=1}^{n}\int_{\R_+}\cF Y(s_{j+1}-s_j,\cdot)(\eta) \times\cF Y(r_{j+1}-r_j,\cdot)(\eta) \notag\\
    &\quad\times|\eta|^{1-2H}  \ud\eta \bigg)J_0(s_1)J_0(r_1)\prod_{j=1}^{n}|s_j-r_j|^{2H_0-2} \ud\mathbf{s}\ud\mathbf{r} \notag\\
    &\ge n!t^{n(2H_0-2)}\mu_0^2\lambda^{2n}C^{n} \int_{\R_+^n}\Bigg|\int_{T_n(t)}\prod_{j=1}^{n}\cF Y(s_{j+1}-s_j,\cdot)(\eta_j)\ud\mathbf{s} \Bigg|^2 \prod_{j=1}^{n}|\eta_j|^{1-2H} \ud\boldsymbol\eta, \label{e:fn-low-b}
  \end{align}
  where we have used Lemma \ref{le:integral-space-nonnegative} and the fact that $|s_j-r_j|\le t$, $J_0(\cdot)\ge\mu_0$ in the last inequality.

  Let
  \begin{equation}\label{e:de-an}
   \begin{split}
     A_n(t) :&= \int_{\R_+^n}\Bigg|\int_{T_n(t)}\prod_{j=1}^{n}\cF Y(s_{j+1}-s_j,\cdot)(\eta_j)\ud\mathbf{s} \Bigg|^2 \prod_{j=1}^{n}|\eta_j|^{1-2H} \ud\boldsymbol\eta \\
     &= \int_{\R_+^n}\Bigg|\int_{T_n(t)}\prod_{j=1}^{n} (s_{j+1}-s_j)^{\beta+\gamma-1}E_{\beta,\beta+\gamma}\left(-2^{-1}\nu (s_{j+1}-s_j)^\beta|\eta_j|^\alpha\right) \ud\mathbf{s} \Bigg|^2 \prod_{j=1}^{n}|\eta_j|^{1-2H} \ud\boldsymbol\eta.
   \end{split}
  \end{equation}
  Using  changes of variables, we have the scaling
  \begin{equation}\label{e:scaling-an}
    A_n(t) = t^{n\left(2\beta+2\gamma-\frac{\beta(2-2H)}{\alpha}\right)} A_n(1).
  \end{equation}
  Now, we treat $\E[A_n(\tau)]$ where $\tau$ is an exponential random variable with parameter $1$. By Jensen's inequality, we get
  \begin{align}
     & \E[A_n(\tau)] =\int_{0}^{\infty}e^{-t} A_n(t)\ud t  \notag\\
     &=\int_{\R_+^n}\int_{0}^{\infty}e^{-t} \Bigg|\int_{T_n(t)}\prod_{j=1}^{n} (s_{j+1}-s_j)^{\beta+\gamma-1}E_{\beta,\beta+\gamma}\left(-2^{-1}\nu (s_{j+1}-s_j)^\beta|\eta_j|^\alpha\right) \ud\mathbf{s} \Bigg|^2 \ud t \prod_{j=1}^{n}|\eta_j|^{1-2H} \ud\boldsymbol\eta \notag\\
     &\ge\int_{\R_+^n}\Bigg|\int_{0}^{\infty}e^{-t} \int_{T_n(t)}\prod_{j=1}^{n} (s_{j+1}-s_j)^{\beta+\gamma-1}E_{\beta,\beta+\gamma}\left(-2^{-1}\nu (s_{j+1}-s_j)^\beta|\eta_j|^\alpha\right) \ud\mathbf{s} \ud t \Bigg|^2 \prod_{j=1}^{n}|\eta_j|^{1-2H} \ud\boldsymbol\eta \notag\\
     &\ge\int_{\left(0,(2/\nu)^{1/\alpha}\right)^n}\Bigg|\int_{0}^{\infty}e^{-t} \int_{T_n(t)}\prod_{j=1}^{n} (s_{j+1}-s_j)^{\beta+\gamma-1}E_{\beta,\beta+\gamma}\left(-2^{-1}\nu (s_{j+1}-s_j)^\beta|\eta_j|^\alpha\right) \ud\mathbf{s} \ud t \Bigg|^2  \notag\\
     &\quad \times\prod_{j=1}^{n}|\eta_j|^{1-2H} \ud\boldsymbol\eta. \label{e:jensen-low-b}
  \end{align}
  Using the change of variables $r_j=s_{j+1}-s_j$, for $j=0,\dots,n$ with the convention $s_0=0$ and $s_{n+1}=t$, we have, when $|\eta|<(2/\nu)^{1/\alpha}$,
  \begin{align}
   & \int_{0}^{\infty}e^{-t} \int_{T_n(t)}\prod_{j=1}^{n} (s_{j+1}-s_j)^{\beta+\gamma-1}E_{\beta,\beta+\gamma}\left(-2^{-1}\nu (s_{j+1}-s_j)^\beta|\eta_j|^\alpha\right) \ud\mathbf{s} \ud t \notag\\
   &=\int_{R_+^{n+1}}e^{-(r_0+\cdots+r_n)} \prod_{j=1}^{n} r_j^{\beta+\gamma-1}E_{\beta,\beta+\gamma}\left(-2^{-1}\nu r_j^\beta|\eta_j|^\alpha\right) \ud\mathbf{r} \notag\\
   &=\prod_{j=1}^{n}\frac{1}{1+2^{-1}\nu|\eta_j|^\alpha},\label{e:integral-time-low-b}
  \end{align}
  where we have used Lemma \ref{le:integral-time-low-b} in the last step. Plugging \eqref{e:integral-time-low-b} into \eqref{e:jensen-low-b}, we have
  \begin{equation*}
    \E[A_n(\tau)]\ge\int_{\left[0,(2/\nu)^{1/\alpha}\right]^n} \prod_{j=1}^{n}\frac{|\eta_j|^{1-2H}}{\left(1+2^{-1}\nu|\eta_j|^\alpha\right)^2}\ud\boldsymbol\eta=c^n,
  \end{equation*}
  where $c\in(0,\infty)$. And then, by \eqref{e:scaling-an}, we have
  \begin{equation}\label{e:antau-an1}
    c^n\le\E[A_n(\tau)]\le\E\left[\tau^{n(2\beta+2\gamma-\frac{\beta(2-2H)}{\alpha})}\right] A_n(1).
  \end{equation}
  Combining \eqref{e:fn-low-b}, \eqref{e:de-an}, \eqref{e:scaling-an}, \eqref{e:antau-an1} and using the fact $\E[\tau^x]=\Gamma(x+1)$, one can see that
  \begin{align*}
     & n!\Vert \tilde{f_n}(\cdot,t,x)\Vert_{\cH^{\otimes n}}^2\ge n!t^{n(2H_0-2)}\mu_0^2\lambda^{2n}C^{n} A_n(t) \\
     & = n!\mu_0^2\lambda^{2n}C^{n}t^{n(2H_0-2)} t^{n(2\beta+2\gamma-\frac{\beta(2-2H)}{\alpha})} A_n(1)\\
     & \ge n!\mu_0^2\lambda^{2n}C^{n}t^{n(2H_0-2)} t^{n(2\beta+2\gamma-\frac{\beta(2-2H)}{\alpha})} \E^{-1}\left[\tau^{n(2\beta+2\gamma-\frac{\beta(2-2H)}{\alpha})}\right]\\
     &= n!\mu_0^2\lambda^{2n}C^{n}t^{n(2H_0-2)} t^{n(2\beta+2\gamma-\frac{\beta(2-2H)}{\alpha})} \Gamma\left(n\left[2\beta+2\gamma-\frac{\beta(2-2H)}{\alpha}\right]+1\right)^{-1}\\
     &\ge \mu_0^2C^{n} \frac {\lambda^{2n} t^{n(2\beta+2\gamma-2-\frac{\beta(2-2H)}{\alpha}+2H_0)}}{(n!)^{2\beta+2\gamma-1-\frac{\beta(2-2H)}{\alpha}}}
     = \mu_0^2C^{n} \frac {\lambda^{2n} t^{n(2H_0\theta+2H_0)}}{(n!)^{2H_0\theta+1}},
  \end{align*}
  where the last step is from \eqref{estimate-gamma}. By \eqref{estimate-sum}, we have
  \begin{align*}
       \Vert u(t,x)\Vert_p &\ge \Vert u(t,x)\Vert_2=\left( \sum_{n\ge0} n!\Vert \tilde{f_n}(\cdot,t,x)\Vert_{\cH^{\otimes n}}^2\right)^{1/2} \ge \mu_0 c_1\exp\left(  c_2|\lambda|^{\frac{2}{2H_0\theta+1}}t^\frac{2H_0(\theta+1)}{2H_0\theta+1} \right).
  \end{align*}
  This proves the theorem.
\end{proof}

\begin{remark}\label{re:dl-compare}
  In this remark, we compare the condition \eqref{e:DL} with some known results:
  \begin{enumerate}
    \item For the stochastic heat equation (i.e., $\alpha=2$, $\beta=1$ and $\gamma=0$),  \eqref{e:DL} is equivalent to \[H_0+H>\frac34,\]   which agrees with \cite[Theorem 1.2]{chen2019parabolic}.
    \item For the fractional wave equation (i.e., $\beta=2$ and $\gamma=0$), \eqref{e:DL} is equivalent to \[\alpha>3-4H,\]   which coincides with \cite[Theorem 3.2]{song2020fractional}.
    \item If the noise $\dot{W}$ is a space-time white noise (i.e., $H_0=H=\frac12$),   \eqref{e:DL} is equivalent to
    \begin{equation*}
      \begin{cases}
          2\alpha+\frac{\alpha}{\beta} \min\left(2\gamma-1,0\right)>1, & \mbox{if } \beta\in(0,2), \\
          \alpha\min(1+\gamma,2)>1, & \mbox{if } \beta=2.
      \end{cases}
    \end{equation*}
    which coincides with \cite[Theorem 1.1]{chen2022moments}.
  \end{enumerate}
\end{remark}

 \begin{remark}
In Step 3 of the proof of Theorem \ref{th:DL}, the  estimation of the lower bound  relies heavily on the nonnegativity of the integral $\int_{\R} \cF Y(r,\cdot)(\eta)\cF Y(s,\cdot)(\eta) |\eta|^{1-2H} \ud \eta$. For the case $\beta\in(0,1]$, this non-negativity holds  as a direct consequence of the non-negativity of the Fourier transform $\cF Y$ of the fundamental solution. For the case $\beta=2, \gamma=0$ (i.e., the wave equation),  though $\cF(t,\cdot) (\eta) =\frac{\sin(t|\xi|^{\alpha/2})}{|\xi|^{\alpha/2}}$ is not a non-negative function, one still has the non-negativity of the integral $\int_{\R} \cF Y(r,\cdot)(\eta)\cF Y(s,\cdot)(\eta) |\eta|^{1-2H} \ud \eta$ (see
Lemma \ref{le:integral-space-nonnegative}). \end{remark}

\begin{proposition}
Assume the same conditions as in Theorem \ref{th:DL}.  If $H_0=\frac{1}{2}$, the condition \eqref{e:DL} is also necessary and the lower bound \eqref{low-b} holds for all $\alpha>0$, $\beta\in(0,2]$ and $\gamma\ge0$.
\end{proposition}

\begin{proof}
  {\em Step 1. Necessity of condition \eqref{e:DL}:} Assuming $u(t,x)$ is a solution in the sense of Definition \ref{def:solution},  clearly for any $n\geq 1$, we have that
 \begin{align*}
\Vert \tilde{f_n}(\cdot,t,x)\Vert_{\cH^{\otimes n}}^2 \le \E[|u(t,x)|^2]<\infty.
 \end{align*}
In particular,  we have for $n=2$, noting $J_0(\cdot)\ge\mu_0$,
\begin{align}\label{e:integral-f2}
\int_{T_2(t)}\int_{\R^{2}}\prod_{j=1}^{2}|\cF Y(s_{j+1}-s_j,\cdot)(\eta_j)|^2\times |\eta_j-\eta_{j-1}|^{1-2H} \ud\mathbf{s}\ud\boldsymbol \eta\le C \Vert \tilde{f_2}(\cdot,t,x)\Vert_{\cH^{\otimes 2}}^2<\infty,
\end{align}
 for some positive constant $C$.
  We bound the integral from below  by  restricting the spatial integration  on $\{(\eta_1, \eta_2)\in \R_+\times \R_-\}$ where  we have
 \begin{equation}\label{e:eta2-eta1}
   |\eta_2-\eta_1|^{1-2H}\geq|\eta_1|^{1-2H}.
 \end{equation}
Noting  \eqref{e:eta2-eta1} and Fubini's theorem,  \eqref{e:integral-f2} holds only if the following integral with respect to $(s_1, \eta_1)$ is finite:
 \begin{align}
       &\int_0^{s_2}\int_{\R_+}|\cF Y(s_2-s_1,\cdot)(\eta_1)|^2|\eta_1|^{2-4H} \ud s_1\eta_1 \notag \\
      =& \int_0^{s_2}\int_{\R_+} (s_2-s_1)^{2\beta+2\gamma-2}E_{\beta,\beta+\gamma}^2\left(-2^{-1}\nu (s_2-s_1)^\beta|\eta_1|^\alpha\right) |\eta_1|^{2-4H} \ud s_1\eta_1\notag\\
      =& \frac{1}{2}\int_0^{s_2}\int_{\R}(s_2-s_1)^{2\beta+2\gamma-2-\frac{\beta(3-4H)}{\alpha}} \left(2^{-1}\nu\right)^{-\frac{3-4H}{\alpha}} E_{\beta,\beta+\gamma}^2\left(-|\eta|^\alpha\right)|\eta|^{2-4H} \ud s_1\ud\eta <\infty \label{con-ne1},
 \end{align}
 where in the last step we have used the change of variables $\eta_1=\left(2^{-1}\nu (s_2-s_1)^\beta\right)^{-\frac{1}{\alpha}}\eta$.
 Then, \eqref{con-ne1} is equivalent to
 \begin{equation}\label{con-ne2}
   \begin{cases}
     2\beta+2\gamma-2-\frac{\beta(3-4H)}{\alpha}>-1, \vspace{0.2cm}\\
   \displaystyle  \int_{\R}E_{\beta,\beta+\gamma}^2\left(-|\eta|^\alpha\right)|\eta|^{2-4H} \ud\eta <\infty.
   \end{cases}
 \end{equation}
 Notice that $E_{\beta,\beta+\gamma}^2\left(-|\cdot|^\alpha\right)|\cdot|^{2-4H}$ is locally integrable. By proposition \ref{E-asymp}, as $\eta\to\infty$,
   \begin{equation*}
    E_{\beta,\beta+\gamma}\left(-|\eta|^\alpha\right)=
  \begin{cases}
   -\frac{1}{\Gamma(\gamma) |\eta|^\alpha}+O(|\eta|^{-2\alpha}),          & \beta\in(0,2),\vspace{0.2cm} \\
   \frac{\cos\left(\sqrt{|\eta|^\alpha}-\pi(\gamma+1)/2\right)}{|\eta|^{\alpha(1+\gamma)/2}} +  \frac{1}{\Gamma(\gamma)|\eta|^\alpha}+O\left(|\eta|^{-2\alpha}\right) , & \beta=2.
   \end{cases}
  \end{equation*}
 Thus, for $\beta\in(0,2)$, \eqref{con-ne2} is equivalent to
 \begin{align*}
  & \begin{cases}
     2\beta+2\gamma-2-\frac{\beta(3-4H)}{\alpha}>-1, \\
     2\alpha-(2-4H)>1,
   \end{cases} \\
  \Longleftrightarrow  &  \, 2\alpha+\frac{\alpha}{\beta} \min(2\gamma-1,0)>3-4H,\\
  \Longleftrightarrow  &  \, 2\alpha+\frac{\alpha}{\beta} \min\left(2\gamma-1,2\gamma-1+\frac{\beta(1-2H)}{\alpha},0\right)>3-4H.
 \end{align*}
For $\beta=2$,  the second condition in \eqref{con-ne2} is equivalent to, for any $\e>0$,
 \begin{equation*}
   \int_{|\eta|>\e} \frac{\cos^2\left(\sqrt{|\eta|^\alpha}-\pi(\gamma+1)/2\right)}{|\eta|^{\alpha(1+\gamma)-(2-4H)}}\ud\eta<\infty \text{ and }
   \int_{|\eta|>\e} \frac{1}{|\eta|^{2\alpha-(2-4H)}} \ud\eta<\infty.
 \end{equation*}
 Therefore, for $\beta=2$, by \cite[Lemma B.1]{chen2022moments}, \eqref{con-ne2} is equivalent to
  \begin{equation*}
    \begin{cases}
      2+2\gamma-\frac{2(3-4H)}{\alpha}>-1, \\
      \alpha \min(1+\gamma,2)-(2-4H)>1,
    \end{cases} \Longleftrightarrow \alpha\min(1+\gamma,2)>3-4H.
  \end{equation*}

  {\em Step 2. Lower bound.} As for the lower bound, by $J_0\ge\mu_0$ we have
  \begin{align*}
       n!\Vert \tilde{f_n}(\cdot,t,x)\Vert_{\cH^{\otimes n}}^2  \ge& \mu_0^2\lambda^{2n}C^{n}\int_{T_n(t)}\int_{\R^{n}}\prod_{j=1}^{n}|\cF Y(s_{j+1}-s_j,\cdot)(\eta_j)|^2\times |\eta_j-\eta_{j-1}|^{1-2H} \ud\mathbf{s}\ud\boldsymbol \eta \\
       \ge  & \mu_0^2\lambda^{2n}C^{n}\int_{T_n(t)}\int_{\D_n} \prod_{j=1}^{n}|\cF Y(s_{j+1}-s_j,\cdot)(\eta_j)|^2\times |\eta_j|^{1-2H} \ud\mathbf{s}\ud\boldsymbol \eta\\
       = & \mu_0^2\lambda^{2n}C^{n}\int_{T_n(t)}\int_{\R^n}\prod_{j=1}^{n}|\cF Y(s_{j+1}-s_j,\cdot)(\eta_j)|^2\times |\eta_j|^{1-2H} \ud\mathbf{s}\ud\boldsymbol \eta,
  \end{align*}
  where in the second inequality we used the fact that $|\eta_j-\eta_{j-1}|^{1-2H}\ge|\eta_j|^{1-2H}$ on $\D_n$ with $\D_n=\{ (\eta_1,\dots,\eta_n)\in\R^n: \eta_1\ge0,\eta_2\le0,\eta_3\ge0,\eta_4\le0,\dots \}$.

  Using Lemma \ref{le:a1} and Lemma \ref{le:t}, we see that
  \begin{equation*}
       n!\Vert \tilde{f_n}(\cdot,t,x)\Vert_{\cH^{\otimes n}}^2  \ge \mu_0^2 C^n\lambda^{2n}\frac{t^{n(\theta+1)}}{\Gamma(n(\theta+1)+1)},
  \end{equation*}
where $\theta=2\beta+2\gamma-2-\frac{\beta(2-2H)}{\alpha}$ when $H_0=\frac12$. And then by Lemma \ref{le2},
  \begin{align*}
       \Vert u(t,x)\Vert_p &\ge \Vert u(t,x)\Vert_2=\left( \sum_{n\ge0} n!\Vert \tilde{f_n}(\cdot,t,x)\Vert_{\cH^{\otimes n}}^2\right)^{1/2} \\
       &\ge \left( \mu_0^2\sum_{n=0}^{\infty}\frac{ C^n\lambda^{2n}t^{n(\theta+1)}}{(n!)^{\theta+1}} \right)^{1/2}\ge \mu_0 c_1\exp\left(  c_2|\lambda|^{\frac{2}{\theta+1}}t \right).
  \end{align*}
  This proves the Proposition.
\end{proof}

\section{H\"older Continuity}\label{se:holder}

In this section, we study the H\"older continuity of the solution $u(t,x)$ to the FDE \eqref{e:fde}.  We first introduce some notations. For two positive constants $a,b$ and a given subset $D\subseteq[0,\infty)\times \R$, let $C_{a,b}(D)$ denote the set of all {\em (locally) H\"older continuous functions over $D$} of order $(a,b)$, namely, given a function $f\in C_{a,b}(D)$,  for each compact subset $K\subseteq D$, there exists a constant $C$ (which may depend on $K$), such that for all $(t,x),(s,y)\in K$,
\begin{equation*}
  |f(t,x)-f(s,y)|\le C \left( |t-s|^a+|x-y|^b \right).
\end{equation*}
We denote \[C_{a^-,b^-}(D):=\bigcap_{\substack{a'\in(0,a)\\b'\in(0,b)}}C_{a',b'}(D).\]


\begin{theorem}\label{th:holder}
Assume $H_0\in[\frac12,1)$, $H\in(0,1/2)$ and the condition \eqref{e:DL}. Let $u(t,x)$ be the mild Skorohod solution to \eqref{e:fde}.  In the case $\beta=2$, for the H\"older continuity in time, we assume the additional condition
\begin{equation}\label{e:DL-add}
 \alpha\gamma>2-2H.
\end{equation}
 Denote
  \begin{equation*}
    \rho:=\beta+\gamma-1-\frac{\beta(1-H)}{\alpha}+H_0,
  \end{equation*}
  and
  \begin{equation}\label{e:kappa}
    \kappa:=
    \begin{cases}
    \alpha-1+H+\frac{\alpha}{\beta}\min(\gamma-1+H_0,0), \ & \mbox{ if } 0<\beta<2,\\
    \frac{\alpha}{2}\min(1+\gamma, 2)-1+H, \ & \mbox{ if } \beta=2.
    \end{cases}
  \end{equation}
   Then, we have
  \begin{equation*}
    u(t,x)\in  C_{\min(\rho,1)^-,\min(\kappa,1)^-}\big([0,\infty)\times\R \big).
  \end{equation*}
\end{theorem}

 \begin{remark}[A comparison of Theorem \ref{th:holder} with known results]\label{re:holder-compare}\hfill

 \begin{enumerate}
   \item  When $\alpha=2$, $\beta=1$ and $\gamma=0$, the equation \eqref{e:fde} is the classical heat equation, then
 \begin{equation*}
   u(t,x)\in C_{\left(H_0+\frac{H}{2}-\frac12\right)^-,\left(2H_0+H-1\right)^-}\big([0,\infty)\times\R \big),
 \end{equation*}
   which  coincides with \cite[Theorem 4.3]{hu2019joint}. If in addition the noise $\dot{W}$ is white in time (i.e., $H_0=\frac12$),  we have
   \begin{equation*}
     u(t,x)\in C_{\frac{H}{2}^-,H^-}\big([0,\infty)\times\R \big),
   \end{equation*}
   which is consistent with \cite[Proposition 3.7]{hu2017stochastic}.

   \item If the noise $\dot{W}$ is a space-time white noise (i.e., $H_0=H=\frac12$),  we have for  $\beta\in(0,2)$,
   \begin{equation*}
     u(t,x)\in C_{\min\left(\beta+\gamma-\frac12-\frac{\beta}{2\alpha},1\right)^-,\min\left(\alpha-\frac12+\frac{\alpha}{\beta}\min(\gamma-\frac12,0),1\right)^-}\big([0,\infty)\times\R \big)
   \end{equation*}
   which is consistent with \cite[Theorem 1]{chen2022holder}.

  \item  When $\beta=2$ and $\gamma\geq0$,  the solution $u(t,x)$ has H\"older continuity in space with order
   $\min\left(\frac{\alpha}{2}\min(1+\gamma, 2)-1+H,1\right)^-$, which agrees with \cite[Proposition 5.1]{song2020fractional} where $\gamma=0$; when $\beta=2$ and $\alpha\gamma>2-2H$,  the order of H\"older continuity in time is improved to $\min\left(1+\gamma-\frac{2-2H}{\alpha}+H_0,1\right)^-$. For the H\"older continuity in time, Theorem \ref{th:holder} does not cover the case of  the fractional stochastic wave equation (i.e. $\beta=2$, $\gamma=0$ and $\alpha\in(0,2]$)  which was considered in \cite{song2020fractional}, due to the  extra condition \eqref{e:DL-add} required by  the new method employed in the proof (see also Remark \ref{rem:explanation}).

  \item As observed in \cite{chen2022holder}, the space H\"older continuity exponent may jump down as $\beta \to 2^{-}$. For instance, if $\gamma=0$, we have $\lim\limits_{\beta\to 2^-} \kappa =\frac{\alpha}{2}(1+H_0)-1+H$ which is strictly bigger than the value $\frac\alpha2 -1+H$ of $\kappa$ for $\beta=2$. In contrast, the time H\"older continuity exponent does not jump as $\beta\to2^-$ because of the additional condition \eqref{e:DL-add} for the case $\beta=2$.

 \end{enumerate}
\end{remark}

The following proposition is a key ingredient in the proof of H\"older continuity in time.

\begin{proposition}\label{prop:coni}
 Assume the same conditions as  in Theorem \ref{th:holder}. Fix $T\in(0,\infty)$, $x\in\R$ and recall $\theta$ given in \eqref{e:theta}.    Then,  for $-1<a\leq 1-2H$ and $0<q<2H_0(\theta+1)$, there exist constants $C>0$ and $\delta\in(0,2H_0)$ such that   \begin{equation}\label{e:icf2h}
 \int_{\R} \left|\cF Y(t-r,\cdot)(\xi)-\cF Y(s-r,\cdot)(\xi)\right|^2 |\xi|^a \ud\xi\le
   C (s\wedge t-r)^{-\delta}|t-s|^{q\wedge2}
  \end{equation}
holds for all $s, t\in (0,T]$ and $0\leq r<s\wedge t$.
\end{proposition}
\begin{proof}
Without loss of generality, we assume $s<t$. Note that $a$ satisfies
     \begin{equation}\label{e:con-a}
  \begin{cases}
    -1<a\leq 1-2H<2\alpha-1, & \mbox{when } \beta\in(0,2) \\
    -1<a\leq 1-2H<\alpha\min(\gamma,2)-1, & \mbox{when } \beta=2.
  \end{cases}
\end{equation}

Recalling \eqref{e:fourier-Y}, we have
  \begin{align*}
  &\int_{\R} |\cF Y(t-r,\cdot)(\xi)-\cF Y(s-r,\cdot)(\xi)|^2 |\xi|^{a}\ud\xi \notag\\
  =&\int_{\R}\left| (t-r)^{\beta+\gamma-1}E_{\beta,\beta+\gamma}(-2^{-1}\nu(t-r)^\beta|\xi|^\alpha)- (s-r)^{\beta+\gamma-1}E_{\beta,\beta+\gamma}(-2^{-1}\nu(s-r)^\beta|\xi|^\alpha)\right|^2|\xi|^{a}\ud\xi.
  \end{align*}
Then by the fundamental theorem of calculus and \eqref{e:dmlf}, the above equation equals
   \begin{equation}\label{e:FY}
   \begin{aligned}
  & \int_{\R} |\cF Y(t-r,\cdot)(\xi)-\cF Y(s-r,\cdot)(\xi)|^2 |\xi|^{a}\ud\xi\\
  =&\int_{\R}\left|\int_s^t (u-r)^{\beta+\gamma-2}E_{\beta,\beta+\gamma-1}(-2^{-1}\nu(u-r)^\beta|\xi|^\alpha)\ud u\right|^2|\xi|^{a}\ud\xi \\
  \le& \left|\int_s^t \bigg(\int_\R (u-r)^{2\beta+2\gamma-4}E_{\beta,\beta+\gamma-1}^2(-2^{-1}\nu(u-r)^\beta|\xi|^\alpha)|\xi|^{a}\ud\xi\bigg)^\frac12 \ud u\right|^2\\
  =& C_{a, \beta+\gamma-1} \left|\int_s^t (u-r)^{\beta+\gamma-2-\frac{\beta(a+1)}{2\alpha}}\ud u\right|^2 \\
  \le& C_{a, \beta+\gamma-1} T^{\frac{\beta(1-2H-a)}{\alpha}} \left|\int_s^t (u-r)^{H_0\theta-1} \ud u\right|^2
  :=I
  \end{aligned}
  \end{equation}
  where the inequality follows from the Minkowski’s inequality, $\theta$ is given in \eqref{e:theta}, and $C_{a, \beta+\gamma-1}$ is finite due to \eqref{e:con-a} and Lemma \ref{le:a1}.

The condition \eqref{e:DL} guarantees $2H_0\theta+1>0$ (see Remark \ref{rem:theta}) and thus  $2H_0(\theta+1)>0$.   We consider the proof in the following three cases.

{\em Case 1: $0<2H_0(\theta+1)\leq1$.} For $0<q<2H_0(\theta+1)\le1$, there exist $\sigma\in(0,\frac12)$ and $\delta\in(0,2H_0)$ satisfying that $2\sigma+\delta$ is close to $2H_0+1$ such that $2\sigma+\delta+2H_0\theta-q-1>0$. Then, the Cauchy-Schwartz inequality yields
\begin{align*}
  I&\le \int_s^t (u-r)^{2H_0\theta-2+2\sigma} \ud u \times \int_s^t (u-s)^{-2\sigma}\ud u\\
   &\le CT^{1-2\sigma}\int_s^t (u-r)^{2H_0\theta-2+2\sigma} \ud u
\end{align*}
 Then, we have
\begin{align*}
  I & \le C \int_s^t (u-r)^{(2\sigma+\delta+2H_0\theta-q-1)-\delta+(q-1)} \ud u \le CT^{2\sigma+\delta+2H_0\theta-q-1}(s-r)^{-\delta}\int_{s}^{t} (u-r)^{q-1} \ud u \\
   &\le C(s-r)^{-\delta}\int_{s}^{t} (u-s)^{q-1} \ud u = C(s-r)^{-\delta} |t-s|^q,
\end{align*}
where the last inequality is due to $q-1\le0$ .

{\em Case 2: $1<2H_0(\theta+1)\le2$.} For $q\in(1,2H_0(\theta+1))$, there exists $\delta\in (0,2H_0)$, such that $2H_0\theta-q+\delta>0$. By the Cauchy-Schwartz inequality
\begin{align*}
  I & \le C |t-s| \int_s^t (u-r)^{(2H_0\theta-q+\delta)-\delta+(q-2)} \ud u \le C T^{2H_0\theta-q+\delta} (s-r)^{-\delta} |t-s| \int_s^t (u-r)^{q-2} \ud u\\
   &\le C (s-r)^{-\delta} |t-s| \int_s^t (u-s)^{q-2} \ud u= C (s-r)^{-\delta} |t-s|^{q}.
\end{align*}

{\em Case 3: $2H_0(1+\theta)>2$.} There exists $\delta\in (0,2H_0)$ such that $2H_0\theta-2+\delta>0$. Then, it is easy to see that
\begin{align*}
  I & \le C|t-s|\int_{s}^{t}(u-r)^{2H_0\theta-2+\delta-\delta}\ud u \le C|t-s|(s-r)^{-\delta} T^{2H_0\theta-2+\delta}\int_s^t 1 \ud u\\
    & \le C|t-s|(s-r)^{-\delta} T^{2H_0\theta-2+\delta}|t-s|= C(s-r)^{-\delta}|t-s|^2.
\end{align*}
This completes the proof.
\end{proof}

Now, we are ready to prove Theorem \ref{th:holder}.

\begin{proof}[Proof of Theorem \ref{th:holder}]

Let $T\in(0,\infty)$ be fixed and $\mathbb K$ be a compact subset of $\R$.  It suffices to show the H\"older continuity in $[0, T]\times \mathbb K$. In the proof of this theorem, the letter $C$ denotes a generic constant independent of $s$, $t$, $x$ and $y$ but maybe depend on $T,\mathbb{K}$, and $C$ might vary from line to line.

{\em Step 1: Increments in space.} For any $(t,x,y)\in [0,T]\times \mathbb{K}^2$, by Minkowski's inequality, \eqref{pnorm-bound},  similar to the calculations in \eqref{e:est-f-n} in  Section \ref{se:ex-un}, we have
\begin{align}\label{eqn-space-I0}
  \Vert u(t,x)-u(t,y)\Vert_p &\le \sum_{n\geq0}\Vert I_n(\tilde{f_n}(\cdot,t,x))-I_n(\tilde{f_n}(\cdot,t,y))\Vert_p  \notag\\
   &\le \sum_{n\geq0}(p-1)^{n/2} \Vert I_n(\tilde{f_n}(\cdot,t,x))-I_n(\tilde{f_n}(\cdot,t,y))\Vert_2,
  \end{align}
  and
  \begin{align}\label{eqn-space-I}
  &\Vert I_n(\tilde{f_n}(\cdot,t,x))-I_n(\tilde{f_n}(\cdot,t,y))\Vert_2^2\notag\\
  =&n!\Vert \tilde{f_n}(\cdot,t,x)-\tilde{f_n}(\cdot,t,y)\Vert_{\mathcal{H}^{\otimes n}}^2\notag\\
  \leq &(n!)^{2H_0-1}J_0^2(T)C^n\lambda^{2n}\bigg(\int_{T_n(t)}\bigg(\int_{\R^n}\prod_{j=1}^{n}|\cF Y(s_{j+1}-s_j,\cdot)(\xi_1+\xi_2+\dots+\xi_j)|^2 |\xi_j|^{1-2H}\notag \\
       & \qquad \qquad \qquad  \qquad \quad \times|e^{-i(\xi_1+\dots+\xi_n)(x-y)}-1|^2 \ud\boldsymbol\xi\bigg)^{\frac{1}{2H_0}} \ud\mathbf{s} \bigg)^{2H_0}\notag\\
 \leq &(n!)^{2H_0-1}J_0^2(T)C^n\lambda^{2n}       \bigg(\int_{T_n(t)}\bigg(\int_{\R^n}\prod_{j=1}^{n}|\cF Y(s_{j+1}-s_j,\cdot)(\eta_j)|^2 |\eta_j-\eta_{j-1}|^{1-2H}\notag \\
       & \qquad \qquad \qquad  \qquad \quad  \times|e^{-i\eta_n(x-y)}-1|^2 \ud\boldsymbol\eta\bigg)^{\frac{1}{2H_0}} \ud\mathbf{s} \bigg)^{2H_0}.
  \end{align}

Using   the following fact: there exists a constant $c$ such that, for any $x\in\R$ and $\varepsilon\in(0,1]$,
 \[
 |1-e^{ix}|^2=2-2\cos x\leq c|x|^{2\varepsilon},
 \]
 we can continue the estimate in \eqref{eqn-space-I} as follows:
 \begin{align}\label{eqn-space-I-2}
  &\Vert I_n(\tilde{f_n}(\cdot,t,x))-I_n(\tilde{f_n}(\cdot,t,y))\Vert_2^2\notag\\
  \leq & (n!)^{2H_0-1}J_0^2(T)C^n\lambda^{2n}\notag\\
  &\times\bigg(\int_{T_n(t)}\bigg(\int_{\R^n}|x-y|^{2\varepsilon}|\eta_n|^{2\varepsilon}\prod_{j=1}^{n}|\cF Y(s_{j+1}-s_j,\cdot)(\eta_j)|^2 |\eta_j-\eta_{j-1}|^{1-2H} \ud\boldsymbol\eta\bigg)^{\frac{1}{2H_0}} \ud\mathbf{s} \bigg)^{2H_0}\notag\\
 \leq & (n!)^{2H_0-1}J_0^2(T)C^n\lambda^{2n}\notag\\
  &  \times \bigg(\int_{T_n(t)}\sum_{a\in \cD_n} \bigg(\int_{\R^n} |x-y|^{2\varepsilon}|\eta_n|^{2\varepsilon}\prod_{j=1}^{n} |\cF Y(s_{j+1}-s_j,\cdot)(\eta_j)|^2 \times |\eta_j|^{a_j} \ud \boldsymbol\eta\bigg)^\frac{1}{2H_0} \ud\mathbf s\bigg)^{2H_0}\notag\\
   = & (n!)^{2H_0-1}J_0^2(T)C^n\lambda^{2n} |x-y|^{2\varepsilon}\notag\\
  &  \times \bigg(\int_{T_n(t)}\sum_{a\in \cD_n} \bigg(\int_{\R^n}|\eta_n|^{2\varepsilon}\prod_{j=1}^{n} |\cF Y(s_{j+1}-s_j,\cdot)(\eta_j)|^2 \times |\eta_j|^{a_j} \ud \boldsymbol\eta\bigg)^\frac{1}{2H_0} \ud\mathbf s\bigg)^{2H_0}.
  \end{align}

Using similar  calculations in \eqref{e:est-f-n}-\eqref{e:time} in Section \ref{se:ex-un}, if we also choose $\varepsilon$ to satisfy the following condition
\begin{equation}\label{delta-1}
  \begin{cases}
    \varepsilon<\alpha-1+H+\frac{\alpha}{\beta}\min(\gamma-1+H_0,0), & \mbox{if } \beta\in(0,2), \\
    \varepsilon<\frac{\alpha}{2}\min(\gamma+1,2)-1+H, & \mbox{if } \beta=2,
  \end{cases}
\end{equation}
 we may apply Lemmas \ref{le:a1}, \ref{le:t} and \ref{le2} to obtain, for $n\ge1$,
 \begin{align}\label{eqn-space-I-2-1}
 & \Vert I_n(\tilde{f_n}(\cdot,t,x))-I_n(\tilde{f_n}(\cdot,t,y))\Vert_2^2
  \notag\\
  \leq & (n!)^{2H_0-1} J_0^2(T) C^n\lambda^{2n} |x-y|^{2\varepsilon}\left(\frac{t^{n(\theta+1)-\frac{\beta\varepsilon}{H_0\alpha}}}{\Gamma\left(n(\theta+1)+1-\frac{\beta\varepsilon}{H_0\alpha}\right)}\right)^{2H_0}\notag\\
    \leq & (n!)^{2H_0-1} J_0^2(T) C^n\lambda^{2n} |x-y|^{2\varepsilon}\left(\frac{T^{n(\theta+1)-\frac{\beta\varepsilon}{H_0\alpha}}}{\Gamma\left(n(\theta+1)+1-\frac{\beta\varepsilon}{H_0\alpha}\right)}\right)^{2H_0}\notag\\
 \leq  &\frac{  J_0^2(T) C^n\lambda^{2n} T^{-\frac{2\beta\varepsilon}{\alpha}}|x-y|^{2\varepsilon}T^{2H_0n(\theta+1)}}{(n!)^{2H_0\theta+1}},
\end{align}
 where we recall $\theta$ is given in \eqref{e:theta}  for $\beta\in(0,2]$.

  Noting that $2H_0\theta+1>0$ by Remark \ref{rem:theta} and combining \eqref{eqn-space-I0}, \eqref{eqn-space-I-2-1} and using Lemma \ref{le2}, we have
 \begin{align*}
   \Vert u(t,x)-u(t,y)\Vert_p\leq C |x-y|^{\varepsilon} \exp\left(C T^{\frac{2H_0(\theta+1)}{2H_0\theta+1}}\right)\le C|x-y|^{\varepsilon}.
      \end{align*}

  By the choice of $\varepsilon$ we know that $\varepsilon<\kappa$ and $\varepsilon\le1$.  The H\"older continuity of $u(t,x)$ in space now follows from Kolmogorov’s continuity criterion.

  {\em Step2. Increments in time.}  For any $(s,t,x)\in [0,T]^2\times \mathbb{K}$,  without loss of generality, we may assume $s<t$.  Similar to \eqref{eqn-space-I0}, by using triangular inequality we obtain
  \begin{equation}\label{e:holdtimesplit}
    \begin{split}
       \Vert u(t,x)-u(s,x)\Vert_p &\leq \mu_1|t-s|+\sum_{n\geq1}(p-1)^{n/2}\left[B_n(s,t)+C_n(s,t)\right]
    \end{split}
  \end{equation}
  where
  \begin{equation*}
  \begin{split}
    B_n(s,t)&=\left\Vert I_n(\tilde{f_n}(\cdot,t,x)\one_{[0,s]^n})-I_n(\tilde{f_n}(\cdot,s,x))\right\Vert_2,\\
    C_n(s,t)&=  \left\Vert I_n(\tilde{f_n}(\cdot,t,x)\one_{[0,t]^n\setminus [0,s]^n})\right\Vert_2.
    \end{split}
  \end{equation*}

First, we assume $s>0$. Now we estimate $B_n(t,s)$, which essentially follows the first part of the proof of Theorem \ref{th:DL}. For $n\ge1$,
\begin{align}
 B_n^2(s,t)\leq & (n!)^{2H_0-1} J_0^2(T) C^n \lambda^{2n}\bigg[\int_{T_n(s)}\bigg(\int_{\R^n}  \prod_{j=1}^{n-1}|\cF Y(r_{j+1}-r_j,\cdot)(\eta_j)|^2 \prod_{j=1}^{n}|\eta_j-\eta_{j-1}|^{1-2H} \notag \\
   & \qquad \qquad \qquad\qquad \qquad\qquad \times  |\cF Y(t-r_n,\cdot)-\cF Y(s-r_n,\cdot)|^2 \ud\boldsymbol\eta\bigg)^{\frac{1}{2H_0}}\ud\mathbf{r}\bigg]^{2H_0}\notag\\
   \leq &(n!)^{2H_0-1}J_0^2(T) C^n \lambda^{2n} \bigg[\sum_{a\in \mathcal D_n}\int_{T_n(s)}\bigg(\int_{\R^n} \prod_{j=1}^{n-1}|\cF Y(r_{j+1}-r_j,\cdot)(\eta_j)|^2 \prod_{j=1}^{n-1} |\eta_j|^{a_j}\notag \\
      & \qquad \qquad \qquad\qquad \qquad\qquad \times |\cF Y(t-r_n,\cdot)-\cF Y(s-r_n,\cdot)|^2 |\eta|^{a_n}\ud\boldsymbol\eta\bigg)^{\frac{1}{2H_0}}\ud\mathbf{r}\bigg]^{2H_0}\notag\\
      \le& (n!)^{2H_0-1}J_0^2(T) C^n \lambda^{2n} \bigg[\sum_{a\in \mathcal D_n} \int_{T_{n}(s)}\bigg(\int_{\R^{n-1}} \prod_{j=1}^{n-1}|\cF Y(r_{j+1}-r_j,\cdot)(\eta_j)|^2 \prod_{j=1}^{n-1} |\eta_j|^{a_j}\ud\boldsymbol\eta \notag\\
      & \qquad \qquad\qquad\qquad \qquad\qquad\times \int_{\R} |\cF Y(t-r_n,\cdot)-\cF Y(s-r_n,\cdot)|^2 |\eta|^{a_n}\ud\eta\bigg)^{\frac{1}{2H_0}}\ud\mathbf{r}\bigg]^{2H_0}.\label{eqn-B-n-0}
\end{align}

 Recalling \eqref{e:a}, we have $a_n\in\{0, 1-2H\}$ for $n\ge 1$, and it follows from Proposition~\ref{prop:coni}  that  there exist constants $C>0$ and $\delta\in(0,2H_0)$ such that for $0<r_n<s<t\leq T$,
  \begin{align}\label{eqn-B-n}
  &\int_{\R} |\cF Y(t-r_n,\cdot)-\cF Y(s-r_n,\cdot)|^2 |\eta|^{a_n}\ud\eta\leq C(s-r_n)^{-\delta}(t-s)^{q\wedge 2}.
  \end{align}

Denote
\[
\epsilon_n=\frac{1}{2H_0}\left(2\beta+2\gamma-2-\frac{\beta(a_n+1)}{\alpha}\right).
\]
For $n=1$, noting that $a_1=1-2H$ and  $\epsilon_1=\theta$,  and thus $n(\theta+1)-\delta/2H_0-\epsilon_n=\theta+1-\delta/2H_0-\epsilon_1>0$. For $n\ge 2$, noting that $\theta+1>0$ by Remark \eqref{rem:theta}, we have
\begin{align*}
&n(\theta+1)-\delta/2H_0-\epsilon_n\ge 2(\theta+1)-\delta/2H_0-\epsilon_n> 2\theta-\epsilon_n+1 \\
&\ge \frac1{2H_0}\left(2\beta+2\gamma-2-\frac{\beta}{\alpha}(3-4H)\right)+1>0
\end{align*}
where the last step follows from  \eqref{e:theta'}. Then, combining \eqref{eqn-B-n-0}-\eqref{eqn-B-n} and using the similar argument as in Section \ref{se:ex-un}, we can obtain
\begin{align*}
 B_n^2(s,t)\leq &(t-s)^{q\wedge 2}(n!)^{2H_0-1}J_0^2(T) C^n \lambda^{2n}\left(\sum_{a\in \mathcal D_n}\frac{t^{n(\theta+1)-\delta/2H_0-\epsilon_n}}{\Gamma\left(n(\theta+1)+1-\delta/2H_0-\epsilon_n\right)}\right)^{2H_0}\notag\\
 \leq &(t-s)^{q\wedge 2}\ \frac{C^n \lambda^{2n}T^{-\delta-2H_0\epsilon_n}T^{2H_0n(\theta+1)}}{(n!)^{2H_0\theta+1}}.
\end{align*}
Then, it follows from \eqref{estimate-sum} that
\begin{align}\label{Bn-est}
\sum_{n\geq 1}(p-1)^{n/2}B_n(s,t)&\leq (t-s)^{\frac q2\wedge 1}C_1\exp\bigg(C_2 T^{\frac{2H_0(\theta+1)}{2H_0\theta+1}}\bigg)\notag\\
&\leq C (t-s)^{\frac q2\wedge 1}.
\end{align}

Next, we treat $C_n(s,t)$. Let $M_{s,t}:=[0,t]^n\setminus[0,s]^n$. Then,
  \begin{equation*}
    M_{s,t}=\bigcup_{\sigma\in S_n}\{(r_1,\dots,r_n):0\leq r_{\sigma(1)}\leq r_{\sigma(2)}\leq\dots\leq r_{\sigma(n)},s\leq r_{\sigma(n)}\leq t\}.
  \end{equation*}
  It follows from the analogous calculations in Section \ref{se:ex-un} that for $n\ge1$,
  \begin{align*}
    C^2_n(s,t) &\leq (n!)^{2H_0-1} J_0^2(T) C^n \lambda^{2n} \\
    &\times\bigg(\sum_{a\in \mathcal D_n}\int_{s}^{t}\ud r_n  \int_{T_{n-1}(r_n)}\bigg( \int_{\R^n} \prod_{j=1}^{n}|\cF Y(r_{j+1}-r_j,\cdot)(\eta_j)|^2 |\eta_j|^{a_j} \ud\boldsymbol\eta\bigg)^{\frac{1}{2H_0}}\ud\mathbf{r}\bigg)^{2H_0} \\
    &\le (n!)^{2H_0-1} J_0^2(T)C^n \lambda^{2n}\bigg(\sum_{a\in \mathcal D_n}\int_s^t (t-r_n)^{\epsilon_n}\frac{r_n^{n(\theta+1)-\epsilon_n-1}}{\Gamma\left(n(\theta+1)-\epsilon_n\right)}\ud r_n\bigg)^{2H_0} \\
    &=(n!)^{2H_0-1} J_0^2(T)C^n \lambda^{2n}\bigg(\sum_{a\in \mathcal D_n}\int_s^t (t-r_n)^{\epsilon_n-\theta} (t-r_n)^{\theta}\frac{r_n^{n(\theta+1)-\epsilon_n-1}}{\Gamma\left(n(\theta+1)-\epsilon_n\right)}\ud r_n\bigg)^{2H_0}\notag\\
   &\leq \frac{J_0^2(T)C^n \lambda^{2n}T^{2H_0n(\theta+1)-2H_0\theta-2H_0} }{(n!)^{2H_0\theta+1}}\left(\int_s^t(t-r_n)^{\theta} \ud r_n\right)^{2H_0} \\
    &= \frac{J_0^2(T)C^n \lambda^{2n}T^{2H_0n(\theta+1)-2H_0\theta-2H_0}}{(n!)^{2H_0\theta+1}} (t-s)^{2H_0(1+\theta)}
  \end{align*}
    where we use the convention $r_{n+1}=t$ and notice that $\epsilon_n-\theta\geq 0$ for $n\ge1$.
 Then, we have
  \begin{align}\label{est-Cn}
\sum_{n\geq 1}(p-1)^{n/2}C_n(s,t)&\leq C(t-s)^\frac{2H_0(1+\theta)}{2}.
     \end{align}

When $s=0$, $B_n(0,t)=0$. Using similar argument, we have, for $n\ge1$,
\begin{equation*}
  C_n^2(0,t)\leq \frac{J_0^2(T)C^n \lambda^{2n}t^{n(2H_0(1+\theta))}}{(n!)^{2H_0\theta+1}} \leq \frac{J_0^2(T)C^n \lambda^{2n}T^{(n-1)(2H_0(1+\theta))}}{(n!)^{2H_0\theta+1}} t^{2H_0(1+\theta)}.
\end{equation*}
Then by Lemma \ref{le2},
\begin{equation}\label{est-Cn2}
  \sum_{n\geq 1}(p-1)^{n/2}C_n(0,t)\leq Ct^\frac{2H_0(1+\theta)}{2}.
\end{equation}

     Using \eqref{e:holdtimesplit}, \eqref{Bn-est}, \eqref{est-Cn} and \eqref{est-Cn2},
we get
\[
\Vert u(t,x)-u(s,x)\Vert_p\leq C|t-s|^{\frac{2H_0(1+\theta)}{2}\wedge 1}, \ \forall (s, t,x)\in [0,T]^2\times \mathbb{K}.
\]
Thus, the H\"older continuity in time follows from Kolmogorov’s continuity criterion.
  The proof of Theorem \ref{th:holder} is concluded.
\end{proof}

\begin{remark} [A comparison of methods dealing with H\"older continuity]\label{rem:explanation}
The approach dealing with the H\"older continuity of the solution to classical SHEs and SWEs (see e.g. \cite{balan2019holder,balan2017hyperbolic,hu2015stochastic,song2017class,song2020fractional}) heavily relies on specific properties of the Fourier transforms of the fundamental solutions of the heat and wave equations, i.e., $\frac{\sin{t|\xi|}}{|\xi|}$ and $e^{-t|\xi|^2/2}$. For general FDEs \eqref{e:fde}, the  Fourier transforms of the fundamental solutions are Mittag-Leffler functions (see \eqref{e:fourier-Y} and \eqref{e:mlf})  which are more involved.  For the case $\beta\in(0,1]$,  the following  {\em complete monotone property}
  \begin{equation*}
    x\in[0,\infty)\mapsto E_{a,b}(-x) \text{ is complete monotone } \Longleftrightarrow 0<a\le1\wedge b,
  \end{equation*}
holds and was used to prove the H\"oder continuity in \cite{chen2019nonlinear}, and for the case $\beta\in(0,2)$,  the techniques of {\em local fractional derivative} and {\em fractional Taylor expansion} was employed in \cite{chen2022holder}.

Our result of Theorem \ref{th:holder} holds for FDE \eqref{e:fde} with $\beta\in(0,2]$ and Proposition~\ref{prop:coni} plays an essential role in the proof. The key in the proof of Proposition \ref{prop:coni} is the utilisation of   \eqref{e:dmlf} in estimating $\int_{\R} |\cF Y(t-r,\cdot)(\xi)-\cF Y(s-r,\cdot)(\xi)|^2 |\xi|^{a}\ud\xi$ (see \eqref{e:FY}), which enables us to cover a wider class of
FDEs  with simpler calculations in comparison with \cite{ chen2022holder,chen2019nonlinear}.   However, this method fails in analysing  the H\"oder continuity for wave equation (i.e., $\beta=2$ and $\gamma=0$) which was studied in \cite{song2020fractional} based on the specific property of $\frac{\sin(t|\xi|^{\alpha/2})}{|\xi|^{\alpha/2}}$ . For instance, applying \eqref {e:FY} for the case $\beta=2, \gamma=0$, we have
  \begin{equation*}
    \begin{split}
       &\int_{\R} |\cF Y(t-r,\cdot)(\xi)-\cF Y(s-r,\cdot)(\xi)|^2 |\xi|^{a}\ud\xi \\
      =&\frac2\nu\int_{\R} \left|\sin\left(\sqrt{\nu/2}(t-r)|\xi|^{\alpha/2}\right)-\sin\left(\sqrt{\nu/2}(s-r)|\xi|^{\alpha/2}\right)\right|^2 |\xi|^{a-\alpha}\ud\xi\\
      =&\int_{\R} \left|\int_s^t\cos\left(\sqrt{\nu/2}(u-r)|\xi|^{\alpha/2}\right)\ud u\right|^2 |\xi|^{a}\ud\xi\\
      \le&   \left|\int_s^t\left(\int_{\R}\cos^2\left(\sqrt{\nu/2}(u-r)|\xi|^{\alpha/2}\right)|\xi|^{a}\ud\xi\right)^\frac{1}{2}\ud u\right|^2.
    \end{split}
  \end{equation*}
Note that $\int_{\R}\cos^2\left(\sqrt{\nu/2}(u-r)|\xi|^{\alpha/2}\right)|\xi|^{a}\ud\xi=\infty$ for all $a\in\R$ and this makes the above estimation invalid.
\end{remark}

\appendix
\section{Some miscellaneous results}\label{ap:lemma}

The following result is borrowed from \cite[Lemma B.3 ]{balan2016intermittency} (see also \cite{mmv01} for the one-dimensional version).
\begin{lemma}\label{le:B.3}
  For any $\varphi\in L^{1/H}(\R^n)$,
  \begin{equation*}
    \int_{\R^n}\int_{\R^n}\varphi(\mathbf t)\varphi(\mathbf{s})\prod_{i=1}^{n}|t_i-s_i|^{2H-2}\ud \mathbf{t}\ud\mathbf{s}\le C_H^n \left( \int_{\R^n}\left|\varphi(\mathbf{t})\right|^{1/H}\ud\mathbf{t} \right)^{2H},
  \end{equation*}
where $C_H>0$ is a constant depending on $H$, and we denote $\mathbf{t}=(t_1,\dots,t_n)$ and $\mathbf{s}=(s_1,\dots,s_n)$.
\end{lemma}

\begin{lemma} \label{le:a1}
 Assume
  \begin{equation}\label{e:con-cg-wd}
  \begin{cases}
    -1<a<2\alpha-1, & \mbox{if } \beta\in(0,2) \\
    -1<a<\alpha\min(1+\gamma,2)-1, & \mbox{if } \beta=2.
  \end{cases}
\end{equation}
Then $C_{a, \beta+\gamma}<\infty$ where  $C_{a,\beta+\gamma}$ is given in \eqref{e:cgamma}. Furthermore, we have
  \begin{equation*}
    \int_{\R} |\cF Y(t,\cdot)(\xi)|^2|\xi|^a   \ud\xi = C_{a, \beta+\gamma} t^{2\beta+2\gamma-2-\frac{\beta (a+1)}{\alpha}}.
  \end{equation*}
\end{lemma}
\begin{proof}
By Proposition \ref{E-asymp}  we have, as $|\xi|\rightarrow\infty$,
  \begin{equation*}
    E_{\beta,\beta+\gamma}\left(-|\xi|^\alpha\right)=
  \begin{cases}
   \displaystyle -\frac{1}{\Gamma(\gamma) |\xi|^\alpha}+O(|\xi|^{-2\alpha}),                                                                                                      & \beta\in(0,2),\vspace{0.2cm} \\
   \displaystyle \frac{\cos\left(\sqrt{|\xi|^\alpha}-\pi(\gamma+1)/2\right)}{|\xi|^{\alpha(1+\gamma)/2}} +  \frac{1}{\Gamma(\gamma)|\xi|^\alpha}+O\left(|\xi|^{-2\alpha}\right) , & \beta=2.
   \end{cases}
  \end{equation*}
Note also that the function $E_{\beta,\beta+\gamma}(-|\xi|^\alpha)$ is continuous in $\xi\in\R$. Clearly, if \eqref{e:con-cg-wd} is satisfied, the constant $C_{a, \beta+\gamma}$ given in \eqref{e:cgamma} is finite. Then, by the change of variable $\xi=\left(2^{-1}\nu t^\beta\right)^{-\frac{1}{\alpha}}\eta$, we have
  \begin{equation*}
    \begin{split}
       \int_{\R}|\cF Y(t,\cdot)(\xi)|^2|\xi|^a \ud\xi &= \int_{\R}t^{2\beta+2\gamma-2} E_{\beta,\beta+\gamma}^2\left(-2^{-1}\nu t^{\beta}|\xi|^\alpha\right)|\xi|^a \ud\xi \\
         &= t^{2\beta+2\gamma-2-\frac{\beta(a+1)}{\alpha}} \left(2^{-1}\nu\right)^{-\frac{a+1}{\alpha}}\int_{\R} E_{\beta,\beta+\gamma}^2\left(-|\eta|^\alpha\right)|\eta|^a \ud\eta\\
         &=C_{a, \beta+\gamma}t^{2\beta+2\gamma-2-\frac{\beta (a+1)}{\alpha}}.
    \end{split}
  \end{equation*}
 This proves the lemma.
\end{proof}

The following result is borrowed from  \cite[Lemma 3.3]{balan2017intermittency}.
\begin{lemma}\label{le:t}
  Let $T_n(t)=\{(t_1,\dots,t_n);0<t_1<\dots<t_n<t\}$ for any $t>0$ and $n\geq1$. Then, for any $b_1,\dots,b_n>-1$, we have:
  \begin{equation*}
    I_n(t,b_1,\dots,b_n):=\int_{T_n(t)} \prod_{j=1}^{n}(t_{j+1}-t_j)^{b_j}dt_1\dots dt_n = \frac{\prod_{j=1}^{n}\Gamma(b_j+1)}{\Gamma(|b|+n+1)}t^{|b|+n},
  \end{equation*}
  where $|b|=\sum_{j=1}^{n}b_j$ and we denote $t_{n+1}=t$. Consequently, if there exist $M>\e>0$ such that $\e\leq b_j+1\leq M$ for all $j=1,\dots,n$, then
  \begin{equation*}
 \frac{c^n}{\Gamma(|b|+n+1)}t^{|b|+n}\le   I_n(t,b_1,\dots,b_n)\leq\frac{C^n}{\Gamma(|b|+n+1)}t^{|b|+n},
  \end{equation*}
  where $c=\inf_{x\in[\e,M]}\Gamma(x)$ and  $C=\sup_{x\in[\e,M]}\Gamma(x)$.
\end{lemma}
\bigskip
The following result is an extension of  \cite [Lemma A.3 ]{song2020fractional} which restricts $b\in[0,1]$.

\begin{lemma}\label{le2}
  For any $a>0$ and $b\in\R$, there exist constants $c,C>0$ depending on $a$ and $b$, such that
  \begin{equation}\label{estimate-gamma}
    c^n(n!)^a \le\Gamma(an+b)\le C^n(n!)^a,
  \end{equation}
  for all $n\in\N$ with $an+b>0$.

 For any $a>0$, there exist constants $c_1,c_2,C_1,C_2>0$ depending on $a$, such that
  \begin{equation}\label{estimate-sum}
    c_1\exp\left(c_2x^\frac{1}{a}\right)\leq \sum_{n=0}^{\infty}\frac{x^n}{(n!)^a}\leq C_1\exp\left(C_2x^\frac{1}{a}\right), \ \forall x>0.
  \end{equation}
\end{lemma}
\begin{proof} The inequalities \eqref{estimate-sum} follow from \cite[Lemma A.3]{song2020fractional}.

 We shall prove
  \begin{equation}\label{e:asym-gamma-anb-na}
    \lim_{n\to\infty}\frac{\Gamma(an+b)}{(n!)^a\: a^{an+b-\frac12}n^{b-\frac12-\frac{a}{2}}(2\pi)^{\frac12-\frac{a}{2}}}=1.
  \end{equation}
 We use the notation $a(x)\sim b(x)$ to indicate that $\frac{a(x)}{b(x)}\to 1$ as $x\to\infty$.
 Using the Stirling’s formula
  \begin{equation*}
  \Gamma(x+1)\sim \sqrt{2\pi x}x^x e^{-x},\ \text{ as } x\to \infty,
  \end{equation*}
  we see that
  \begin{align}\label{e:asym-an-b}
  &\Gamma(an+b)=\frac{1}{an+b}\Gamma(an+b+1)\sim \frac{1}{an+b}\sqrt{2\pi (an+b)}\,(an+b)^{an+b} e^{-(an+b)}\notag\\
  &\sim \frac{1}{an}\sqrt{2\pi an}\, (an+b)^{an} (an+b)^be^{-(an+b)} \sim \frac{1}{an}\sqrt{2\pi an}\, (an)^{an} e^{b}(an)^be^{-(an+b)}\notag\\
 &  \sim \frac{1}{an}\sqrt{2\pi an}\, (an)^{an} (an)^be^{-an}.
  \end{align}
  Note also that
  \begin{equation}\label{e:asym-gamma-n}
    n!=\Gamma(n+1)\sim \sqrt{2\pi n}n^n e^{-n}.
  \end{equation}
  From \eqref{e:asym-an-b} and \eqref{e:asym-gamma-n}, we obtain
  \begin{equation*}
       \Gamma(an+b)\sim  (n!)^a a^{an+b-\frac12}n^{b+\frac12-\frac{a}{2}}(2\pi)^{\frac12-\frac{a}{2}},
  \end{equation*}
  which is \eqref{e:asym-gamma-anb-na}.

  It implies from the limit in \eqref{e:asym-gamma-anb-na} that there exists two constants $L_1\in (0,1)$ and  $ L_2>1$ such that
  \begin{equation}\label{eqn-1}
  L_1 (n!)^a a^{an+b-\frac12}n^{b+\frac12-\frac{a}{2}}(2\pi)^{\frac12-\frac{a}{2}}\leq  \Gamma(an+b)\leq L_2 (n!)^a a^{an+b-\frac12}n^{b+\frac12-\frac{a}{2}}(2\pi)^{\frac12-\frac{a}{2}}, \ \forall n\in \mathbb{N}.
  \end{equation}
Then, there must exist some $C>1$ and $0<c <1 $ such that
\begin{equation}\label{eqn-2}
L_2  a^{an+b-\frac12}n^{b+\frac12-\frac{a}{2}}(2\pi)^{\frac12-\frac{a}{2}}\leq C^n,  \ \forall n\in \mathbb{N}
\end{equation}
and
\begin{equation}\label{eqn-3}
L_1  a^{an+b-\frac12}n^{b+\frac12-\frac{a}{2}}(2\pi)^{\frac12-\frac{a}{2}}\geq c^n, \ \forall n\in \mathbb{N}.
\end{equation}
Finally, we combine \eqref{eqn-1}-\eqref{eqn-3} to obtain \eqref{estimate-gamma}.
\end{proof}

The following result is an extension of  \cite [Lemma A.5]{song2020fractional}.

\begin{lemma}\label{le:integral-space-nonnegative}
  Assume that either $0\leq\beta\le1$ or $\beta=2$ and $\gamma=0$. Then, under \eqref{e:DL}, for $H\in(0,1/2)$ and $r,s>0$, we have
  \begin{equation*}
  0\le  \int_{\R} \cF Y(r,\cdot)(\eta)\cF Y(s,\cdot)(\eta) |\eta|^{1-2H} \ud \eta < \infty.
  \end{equation*}
\end{lemma}
\begin{proof}
  When $\beta\in(0,1]$, we have $\cF Y(t,\cdot)(\eta)=t^{\beta+\gamma-1}E_{\beta,\beta+\gamma}(-2^{-1}\nu t^\beta|\eta|^\alpha)\ge0$ (see \cite[Section 4.10.2]{gorenflo2020mittag}).
  Moreover, condition \eqref{e:DL} implies
\begin{equation*}
     2(1-2H)<2\alpha-1.
   \end{equation*}
Then, in this case we use similar arguments in Lemma \ref{le:a1} with $a=1-2H$ to obtain
\[
\int_{\R} |\cF Y(r,\cdot)(\eta)\cF Y(s,\cdot)(\eta)| |\eta|^{1-2H} \ud \eta < \infty.
\]

  When $\beta=2$ and $\gamma=0$, condition \eqref{e:DL} implies
  $ 2(1-2H)<\alpha-1$, which yields $\alpha+2H-2>1-2H>0$.
   Then
  \begin{align}\label{Fourier-Y-nonneg}
    &\int_{\R} \cF Y(r,\cdot)(\eta)\cF Y(s,\cdot)(\eta) |\eta|^{1-2H} \ud \eta\notag \\
    &=\frac{2}{\nu} \int_{\R} \sin\left(\sqrt{\nu/2} r|\eta|^\frac\alpha 2 \right)\sin\left( \sqrt{\nu/2}s|\eta|^\frac\alpha 2\right)|\eta|^{1-2H-\alpha}\ud\eta \notag\\
    &=\frac{1}{\nu}\int_{\R}\left[\cos\left(\sqrt{\nu/2}(r-s)|\eta|^\frac\alpha2 \right)-\cos\left(\sqrt{\nu/2}(r+s)|\eta|^\frac\alpha2 \right) \right]|\eta|^{1-2H-\alpha}\ud\eta\notag\\
    &=\frac{1}{\nu} \int_{\R}\left(\left[1-2\sin^2\left(\sqrt{\nu/2}(r-s)|\eta|^\frac\alpha2 /2\right) \right]-\left[1-2\sin^2\left(\sqrt{\nu/2}(r+s)|\eta|^\frac\alpha2 /2\right) \right]\right)|\eta|^{1-2H-\alpha}\ud\eta \notag\\
    &=\frac{2}{\nu} \int_{\R}\sin^2\left(\sqrt{\nu/2}(r+s)|\eta|^\frac\alpha2 /2\right)|\eta|^{1-2H-\alpha}\ud\eta -\frac{2}{\nu}\int_{\R}\sin^2\left(\sqrt{\nu/2}(r-s)|\eta|^\frac\alpha2 /2\right)|\eta|^{1-2H-\alpha}\ud\eta\notag\\
    &=\frac{2}{\nu} \left(\bigg|\frac{r+s}{2}\bigg|^{\frac{2}{\alpha}(\alpha+2H-2)}-\bigg|\frac{r-s}{2}\bigg|^{\frac{2}{\alpha}(\alpha+2H-2)}\right) \int_{\R} \sin^2\left(\sqrt{\nu/2}|\eta|^\frac\alpha2\right)|\eta|^{1-2H-\alpha}\ud\eta\ge0,
  \end{align}
  where we have used the change of variables in the last equality and $\frac{2}{\alpha}(\alpha+2H-2)>0$. By changing variables, we get
  \begin{equation}\label{eqn-sine}
  \begin{aligned}
  &\int_{\R} \sin^2\left(\sqrt{\nu/2}|\eta|^\frac\alpha2\right)|\eta|^{1-2H-\alpha}\ud\eta=2\int_0^\infty \sin^2\left(\sqrt{\nu/2}\eta^\frac\alpha2\right)\eta^{1-2H-\alpha}\ud\eta\\
  &=\frac{4}{\alpha}\left(\frac{\nu}{2}\right)^{\frac{1}{\alpha}(\alpha+2H-1)}\int_0^\infty\sin^2(\eta)\eta^{-\left(\frac{2}{\alpha}(\alpha+2H-2)+1\right)}\ud \eta<\infty,
  \end{aligned}
  \end{equation}
 where the finiteness follows from \cite[Lemma A.4 ]{song2020fractional}  noting $\frac{2}{\alpha}(\alpha+2H-2)+1\in (1,3)$. Combining \eqref{Fourier-Y-nonneg} and  \eqref{eqn-sine}, one can obtain the desired result.
  \end{proof}

\begin{lemma}\label{le:integral-time-low-b}
  For any $0<\eta<(2/\nu)^{1/\alpha}$, we have
  \begin{equation*}
    \int_{0}^{\infty}e^{-r} r^{\beta+\gamma-1}E_{\beta,\beta+\gamma}\left(-2^{-1}\nu r^\beta|\eta|^\alpha\right) \ud r=\frac{1}{1+2^{-1}\nu|\eta|^\alpha}.
  \end{equation*}
\end{lemma}
\begin{proof}
 Using \eqref{e:mlf}, we have
  \begin{align*}
     & \int_{0}^{\infty}e^{-r} r^{\beta+\gamma-1}E_{\beta,\beta+\gamma}\left(-2^{-1}\nu r^\beta|\eta|^\alpha\right) \ud r \\
     &= \sum_{k=0}^{\infty}\left(-2^{-1}\nu|\eta|^\alpha\right)^k \frac{1}{\Gamma(\beta k+\beta+\gamma)}\int_{0}^{\infty}r^{\beta k+
     \beta+\gamma-1}e^{-r}\ud r\\
     &= \sum_{k=0}^{\infty}\left(-2^{-1}\nu|\eta|^\alpha\right)^k=\frac{1}{1+2^{-1}\nu|\eta|^\alpha}.
  \end{align*}
  The proof is completed.
\end{proof}

\noindent{\bf Acknowledgements.}
J. Song is partially supported by National Natural Science Foundation of China grant 12071256, and Major Basic Research Program of the Natural Science Foundation
of Shandong Province in China ZR2019ZD42 and ZR2020ZD24.

\bibliographystyle{plain}
\bibliography{bibliography}
\end{document}